\documentclass[10pt]{amsart} 
\usepackage{amsmath, amsfonts, amsthm, amssymb, amscd}
\usepackage[mathcal]{euscript}  
\usepackage{dsfont, mathrsfs, accents, verbatim,a4wide} 
\usepackage{bbm}
\usepackage{graphicx,caption,subcaption}
\usepackage{multirow}
\usepackage{enumerate}
\usepackage[foot]{amsaddr}
\theoremstyle{plain}
\newtheorem{theorem}{Theorem}[section]
\newtheorem{proposition}[theorem]{Proposition} 
\newtheorem{lemma}[theorem]{Lemma}
 
\newtheorem{definition}[theorem]{Definition} 
\newtheorem{remark}[theorem]{Remark}   
\numberwithin{equation}{section}      
\let\oldmarginpar\marginpar
\renewcommand\marginpar[1]{\-\oldmarginpar[\raggedleft\footnotesize #1]%
{\raggedright\footnotesize #1}}
\newcommand{\R}{{\mathbb{R}}}

\title[Nonlocal conservation laws]{Nonlocal conservation laws. I. 
\\
A new class of monotonicity-preserving models} 
 
\author[Q. Du, Z. Huang, and P.G. L{\scriptsize e}Floch]{Qiang Du$^{1,2}$, Zhan Huang$^2$, and Philippe G. L{\scriptsize e}Floch$^3$}

\address{$^1$ Department of Applied Physics and Applied Mathematics, Columbia University, New York, NY10027,  USA. Email: {\sl qd2125@columbia.edu}}
\address{ 
$^2$ Department of Mathematics,  Pennsylvania State University, University Park, PA 16802, USA. Email: {\sl zxh117@psu.edu}}
\address{
$^3$ Laboratoire Jacques-Louis Lions and Centre National de la Recherche Scientifique, 
Universit\'e Pierre et Marie Curie, 4 Place Jussieu, 75252 Paris, France. 
Email : {\sl contact@philippelefloch.org.}
\\
\hfill 
\date{November 2016}}

\begin{document}

\begin{abstract} We introduce a new class of nonlocal nonlinear conservation laws in one space dimension
that allow for nonlocal interactions over a {\em finite horizon}. The proposed model, which we refer to as the {\em nonlocal pair interaction model}, inherits at the continuum level  the unwinding feature of finite difference schemes for local hyperbolic conservation laws, so that the maximum principle and certain monotonicity properties hold and, consequently, the entropy inequalities are naturally satisfied. We establish a global-in-time well-posedness theory for these models which covers a broad class of initial data.  Moreover, in the limit when the horizon parameter approaches zero, we are able to prove that our nonlocal model reduces to the conventional class of local hyperbolic conservation laws. Furthermore, we propose a numerical discretization method adapted to our nonlocal model, which relies on a monotone numerical flux and a uniform mesh, and we establish that these numerical solutions converge to a solution, providing as by-product both the existence theory for the nonlocal model and the convergence property relating the nonlocal regime and the asymptotic local regime.  
\end{abstract}

\maketitle 


 

\section{Introduction}

\subsection{Objective of this paper}

Our aim in this paper and its companion \cite{BDL} is to propose and analyze new models of nonlocal conservation laws. We are especially interested in distinguishing between models that preserve the maximum principle (considered in the present work) and models that violate this principle (to be discussed in \cite{BDL}). We are going to introduce here a new 
model, we refer to as the ``nonlocal pair-interaction model'' and we will prove that the associated initial value problem for this model is well-posed in a class of weak solution.   

Recall that scalar one-dimensional hyperbolic conservation laws of the form
 \begin{equation}\label{eq:local_ConservationLaw}
 \begin{array}{l l}
 & \displaystyle u_t + f(u)_x =0, \quad\quad \quad (x,t)\in{\R}\times[0,T], \\[.1cm]
 &u(x,0)=u_0(x),  \quad\quad \quad\quad x\in {\R},\\
 \end{array}
\end{equation}
 have been thoroughly investigated,  both
analytically and numerically. It is well known that even if the
initial value $u_0=u_0(x)$ is smooth, the local-in-time solution may develop shock discontinuities in finite time. The equation \eqref{eq:local_ConservationLaw} should be understood in a suitably weak sense and a so-called entropy condition must be impose din order to select the physically relevant weak solutions.

On the other hand, for the numerical computation of these weak solutions, a variety of algorithms have been
proposed in the literature, among which one of the most fundamental classes is the class of three-point schemes in conservation form, that is, with standard notation
\begin{equation}\label{eq:FD_local}
u^{n+1}_{j}=u^n_{j}-\frac{\Delta t}{\Delta
x} \big(
g(u^n_{j},u^n_{j+1})-g(u^n_{j-1},u^n_{j})
\big),
\end{equation}
where the given numerical flux $g$ is a Lipschitz continuous function consistent
with the flux $f=f(u)$, in the sense that $g(u,u)=f(u)$. It  is well known that if $g=g(u,v)$ is
a monotone flux, i.e, is non-decreasing in $u$ and
non-increasing in $v$, then the scheme
(\ref{eq:FD_local}) is Total Variation Diminishing 
(TVD) and enjoys the Maximum Principle \cite{crandall1980monotone, leveque1992numerical}. 
By picking appropriate choice of $g$ satisfying the above monotonicity property, one obtains numerical solutions converging to an entropy weak solution to the (local) conservation law (\ref{eq:local_ConservationLaw}). 

Furthermore, nonlocal models arising in the literature violates the maximum principle, which seems to be a natural property to be imposed on a (scalar) physical model.  Our idea therefore is to built upon the experience acquited with numerical scheme in order to propose a new model that share many properties of  local conservation laws.  
With these motivations in mind, we thus propose here the following nonlocal generalization to  (\ref{eq:local_ConservationLaw}):
\begin{equation}
\label{eq:Cauchy}
\aligned
 \frac{\partial u}{\partial t} 
+\int^{\delta}_{0}\left(\frac{g(u,\tau_{h}u)-g(\tau_{-h}u,u)}{h}\right)\omega^{\delta}(h)dh
& =0,
  \quad &&(x,t)\in{\R}\times[0,T], 
\\
u(x,0)&=u_0(x),  \quad\quad  \quad &&x\in {\R}, 
\endaligned
\end{equation}
where $\tau_{\pm h} u(x,t)=u(x\pm h,t)$ denotes the shift operator.
We refer to 
(\ref{eq:Cauchy}) as the {\bf nonlocal pair-interaction model} which is uniquely 
determined by
 a choice of kernel $\omega^\delta
=\omega^\delta(h)$ characterizing  {\em nonlocal interactions} and
a nonlocal two-point  flux $g=g(u,v)$ (defined over pairs of points $u,v$). 

For definiteness, we assume that the nonlocal interaction
kernel $\omega^\delta =\omega^\delta(h)$  is a compactly-supported, symmetric, nonnegative density function, that is, more specifically, it is symmetric $\omega^{\delta}(-s)=\omega^{\delta}(s)$, 
nonnegative $\omega^{\delta}(s)\geq 0$, integrable $\omega^{\delta}\in L^{1}(\R)$, 
normalized so that $\|\omega^{\delta}\|_{L^1(\R)}=1$, 
and supported in some interval $[-\delta,\delta]$ (with $\delta>0$). The parameter $\delta$ is refered to as the nonlocal {\em horizon} measuring the range of interactions.
The associated nonlocal flux $g=g(u,v)$ is required to be a monotone function (over pairs of points $u,v$), consistent with the exact flux $f$.

 Formally,  equation (\ref{eq:Cauchy}) may be seen as a {\em continuum average} of
the conservative finite difference scheme (\ref{eq:FD_local}), in the sense
that the pure difference approximation with a fixed grid size is replaced by an integral of
weighted differences, with the weights given by the kernel $\omega^{\delta}=\omega^{\delta}(s)$,
on a continuum scale  up to the horizon
parameter $\delta>0$.  As a result,
the equation (\ref{eq:Cauchy}) is spatially {\em nonlocal}, that is, the pair-interaction
between points $x$ and $y$ is allowed as long as their pairwise
distance is no larger than $\delta$. In contrast, the standard
conservation law (\ref{eq:local_ConservationLaw}) is {\em local}, since
the derivative of flux $f(u)_x$ implies that interaction happens
only within infinitesimal distances (say, by contact in the language of material science).

\subsection{Background on nonlocal models}

In the existing literature, nonlocality has been introduced in various ways for the modeling of convection problems. 
For example,  many fractional convection-diffusion models involve fractional convection terms and 
include a (possibly generalized) diffusive term; see, e.g., Ervin, Hewer and Roop \cite{ervin2007numerical}, Biler and Woyczy\'{n}ski \cite{biler1998global}, Woyczy\'{n}ski \cite{woyczynski1998burgers}, and Mi\u{s}kinis \cite{mivskinis2002some}.
Alternatively, e.g., Dronio \cite{droniou2005fractal} and Alibaud et. al. \cite{alibaud2010non, alibaud2010asymptotic},
  the local convection operators are retained, and nonlocality is introduced through a fractional derivative operator that modifies 
         the diffusive term. 

More general 
        nonlocal regularization terms can be found in spectral viscosity methods introduced by Tadmor \cite{tadmor, cdt93}. Other
        nonlocal regularizations can be found in 
 \cite{schmeiser2004burgers}, Liu \cite{liu2006wave}, Chmaj \cite{chmaj2007existence}, Duan,Feller and Zhu \cite{duan2010energy}, Rohde \cite{rohde2005scalar}, Kissling and Rohde \cite{kissling2010computation}, and Kissling, LeFloch and Rohde \cite{kissling2009kinetic}. 
Nonlocal convection may also be introduced through a nonlocal regularization of the convective velocity, that is, for
a transport equation of the form   $u_t+(vu)_x=0$, we may have $v$ being an integral average of some function of $u$,
see 
for instance Zumbrun \cite{zum99}, Logan \cite{david2003nonlocal}
and Amorim, Colombo and Teixeira
 \cite{aea15}.
 More general nonlocal flux has been studied in Al\`{\i}, Hunter and Parker \cite{ali2002hamiltonian}, Benzoni-Gavage \cite{benzoni2009local}, leading to equations of the form
\begin{equation}
\label{eq:nlf}
u_t+\mathcal{F}[u]_x=0,\quad \quad 
\hat{\mathcal{F}}[u](k)=\int_{\R}\Gamma(k-l)\hat{u}(k-l)\hat{u}(l)dl, 
\end{equation}
 where the hat symbol denotes the Fourier transform.
 Ignat and Rossi \cite{ignat2007nonlocal} analyzed a linear nonlocal evolution equation that allows 
 convective effect. A similar investigation can be found in Du, Huang and Lehoucq \cite{du2014nonlocal}.  
Among earlier works, perhaps the study most closely related to our work is the one by Du, Kamm, Lehoucqs, and Parks \cite{du2012new} who analyzed a nonlocal evolution equation but in a form different from the one we propose here. In fact, the model in \cite{du2012new} failed to preserve the maximum principle which motivated us to consider better alternatives, such as 
\eqref{eq:Cauchy}. We note also that for a quadratic nonlinear flux,
both the models in \cite{du2012new} and in this paper can also be
formulated in the form of \eqref{eq:nlf} but with more general kernels $G$ which are  no longer translation invariant.

For further observations and references, we refer to our companion work \cite{BDL} which will analyze the connections between the above models and our new model \eqref{eq:Cauchy} and its variants.

\subsection{Properties of the proposed model} 

A special choice of the kernel $\omega^{\delta}$ in the nonlocal conservation law (\ref{eq:Cauchy})
allows us to recover the standard one (\ref{eq:local_ConservationLaw}),  
so our nonlocal conservation law is truly an extension of its
local counterpart.  Moreover, one may argue that nonlocal conservation laws are in fact more physical  continuum models, as they enjoy the {\sl unwinding feature} observed at the continuum level and this property is well-known to be essential (e.g. in order to prevent oscillations in numerical solutions) when the underlying solutions may not be smooth. 
The monotone and upwind flux construction implies naturally that the maximum principle remains
valid for the nonlocal model. Moreover, the classical entropy condition is also automatically satisfied by solutions of\ nonlocal models.

Our study of (\ref{eq:Cauchy}) carried out on the present paper includes the following contributions: a
 detailed model set-up for  (\ref{eq:Cauchy}) and its relation with the nonlocal conservation laws
  in the local model; 
   regularity  properties and possible development of shocks from smooth initial data; 
the convergence of numerical discretization of  (\ref{eq:Cauchy})
for a given $\delta$ but with increasing numerical resolution;
and the asymptotical compatibility of the discrete nonlocal solution  \cite{td14}
to the local entry solution when both $\delta$ and the discretization parameter aproach zero. 
As far as the discretization of the nonlocal model is concerned, we propose here a monotone scheme  (\ref{eq:nonlocal_ForwardInTime}) satisfying the maximum principle as well as the TVD stability property.  
Our main convergence result, in Theorem \ref{thm_convergence_local_nonlocal_entropy_solution}, 
 establishes that, under the nonlocal CFL condition and 
 when $\delta$ is fixed and $\Delta x\rightarrow 0$ (and thus $\Delta t\to 0$),  the numerical solution of (\ref{eq:nonlocal_ForwardInTime}) converges to the solution of the nonlocal conservation law. As $\delta$ and $\Delta x$ (and therefore $\Delta t$) tend to zero, the numerical solution converges to the unique entropy solution of the limiting local conservation law.  This study  also provides us with the uniqueness and existence of the solution to the nonlocal conservation law; see the statements in Theorems \ref{thm_uniqueness} and  \ref{corollary_existence_nonlocal_entropy_solution}, respectively.


\subsection{The relation with monotone schemes}
\label{sec:def}

We now discuss some properties of our nonlocal model (\ref{eq:Cauchy}) and compare with the class of discrete schemes. The flux  $g$ in (\ref{eq:FD_local}) may take on different
forms, as the case for the numerical solution of local hyperbolic conservation laws.
A classical example is the Godunov scheme \cite{godunov1959difference}, while another standard example is the
Lax-Friedrich scheme, expressed in the conservative form
(\ref{eq:FD_local}) with
\begin{equation}
g(u^n_j,u^n_{j+1})=\frac{1}{2}(f(u^n_j)+f(u^n_{j+1}))-\frac{\Delta
x}{2\Delta t}(u^n_{j+1}-u^n_{j}).
\end{equation}
It is convenient to recall the following standard notions: 
are given below:
\begin{itemize}
\item A conservative scheme, by definition, can be written into the form (\ref{eq:FD_local}).
\item A scheme is consistent if the flux $g$ in (\ref{eq:FD_local}) satisfies $g(u,u)=f(u)$. 
\item It is said to enjoy the maximum principle if $
\min_{k}u^0_{k}\leq u^n_j\leq\max_{k}u^0_{k}$, for all $n,j$.
\item It is said to be monotone  if 
$u^0_j\leq v^0_j$  for all $j$ implies 
$u^n_j\leq v^n_j$ for all $n, j$.

\item It is said to be Total Variation Diminishing (TVD)
if $
\sum_{j}|u^{n+1}_{j+1}-u^{n+1}_j|\leq\sum_{j}|u^n_{j+1}-u^n_{j}|$.
\end{itemize}

 It is well known that if $g$ is monotone, i.e,
non-decreasing on the first argument, and non-increasing on the
second argument, then under the local CFL condition (for all relevant values $a,b$) 
\begin{equation}
\frac{\Delta t}{\Delta x}\left(\left|\frac{\partial g(a,b)}{\partial
a}\right|+\left|\frac{\partial g(a,b)}{\partial b}\right|\right)\leq
1,
\end{equation}
then \eqref{eq:FD_local} is a conservative and monotone scheme.
Moreover, a conservative monotone scheme is automatically monotonicity preserving, 
maximum principle preserving and TVD \cite{leveque1992numerical}.  

A number of early constructions of approximate solutions to scalar conservation laws 
utilizes the monotonicity property \cite{crandall1980monotone, sanders1983convergence}. We point out that  monotone approximations are limited to first-order accuracy \cite{harten1976finite}.  
On the other hand, a less restrictive class of scheme is obtained by requiring that the TVD property holds. 
We refer to Ol\u{e}inik \cite{oleinik1957discontinuous}, Vol'pert \cite{vol1967spaces}, Kru\u{z}kov \cite{kruvzkov1970first}, and Crandall \cite{crandall1972semigroup} for analysis of solutions with finite total variation and, more generally, solutions satisfying the $L^1$ contraction property. 
 In one dimension, the TVD property enables the construction of convergent difference schemes of high-order resolution, as was initiated by Harten in \cite{harten1983high} and subsequently discussed
by many other researchers.  Our current work is based on the class of monotone schemes (with first-order accuracy), which provides the simplest setup for the development of the new theory of nonlocal conservation laws we propose, but, clearly, higher-order TVD models could be defined similarly and would provide an interesting extension to the present study.

\subsection{Summary of this paper}

In summary, a class of nonlocal  conservation laws is proposed here as a generalization or relaxation
of the conventional local conservation law.  While many forms of nonlocal relaxations have been studied
before in the literature, our nonlocal conservation law has novel and attracting features
and, importantly, shares many important properties of the local conservation law, especially the conservation property and the maximum principle. Our nonlocal model also enjoys a generalization of the entropy condition, and hence provides physically sensible solutions. Moreover, 
the model reduces to its local counterpart  in the local limit.  By designing a monotone scheme adapted to the nonlocal model, we offer a constructive proof to the existence and uniqueness of entropy solutions.
Moreover,  the solution to a given Cauchy problem, as the horizon  $\delta$ is fixed and $\Delta x \rightarrow 0$,  converges to the entropy solution of the nonlocal conservation law, while as both $\delta$ and $\Delta x$ vanish, converges to  the entropy solution of the local conservation law. 
This leads to the so-called asymptotic compatibility, as defined in \cite{td14}, of the discrete schemes for
the nonlocal model. Further investigations, such as the shock formation in the nonlocal conservation models, as well as numerical simulations could be carried out in order to further explore the properties of the solutions. We refer to the companion work \cite{BDL} for variants of \eqref{eq:Cauchy} and the discussion of related issues such as the existence and properties of traveling wave solutions, as well as the connections with other nonlocal models.


\section{The nonlocal pair-interaction model}

\subsection{Nonlocal interaction kernel}

We begin with some specific details about the nonlocal conservation law of interest.
First, the nonlocal interaction kernel
$\omega^{\delta}: \R\rightarrow[0,+\infty)$ is a nonnegative
density, supported in $(0,\delta)$: 
\begin{equation}\label{eq:w_density}
\omega^{\delta}\geq 0, \qquad \omega^{\delta} \mbox{ is supported on } [0,\delta],\quad\quad \int_{0}^{\delta}\omega^{\delta}(h)dh=1.
\end{equation}
Without essential restriction, we can choose
 $\omega^{\delta}(h)=\frac{1}{\delta}\rho\left(\frac{h}{\delta}\right)$, where $\rho$ is a non-negative density function supported on
$[0,1]$, and we assume $\omega^{\delta}\in C^2(0,\delta)$.

 The flux $g=g(u_1,u_2)$ defined for $u_1, u_2 \in \R$ is assumed to satisfy the following
conditions:\\
 (i) $g$ is \emph{consistent
} with  a local flux $f$: 
\begin{equation}\label{eq:g_consistent}
 g(u,u)=f(u).
\end{equation} 
(ii) $g: W^{1,\infty}(\R)\times W^{1,\infty}(\R)\rightarrow
W^{1,\infty}(\R)$ and its partial derivatives, denoted as $g_1$ and $g_2$,  are Lipschitz continuous, with Lipschitz constant $C$: 
\begin{equation}\label{eq:g_Lip}
||g(a,b)-g(c,d)||_{\infty}+\sum_{i=1}^{2}||g_i(a,b)-g_i(c,d)||_{\infty}\leq
C(||a-c||_{\infty}+||b-d||_{\infty}).
\end{equation}
(iii) $g$ is 
nondecreasing with respect to the first argument, and nonincreasing
to the second argument:
\begin{equation}\label{eq:g_monotone}
 g_1(u_1,u_2):=\frac{\partial g}{\partial u_1}(u_1,u_2)\geq 0,\quad
g_2(u_1,u_2):=\frac{\partial g}{\partial u_2}(u_1,u_2)\leq 0.
\end{equation}
(iv) The partial derivatives $g_i$ are locally bounded functions in $L^\infty$, with 
 \begin{equation} \label{eq:gi_bded}
||g_i(a,b)||_{\infty}\leq C(||a||_{\infty}+||b||_{\infty}).
\quad\quad i=1,2,
\end{equation}
for some constant $C$. 

We define the operator $\mathcal{L}^\delta$  by
\begin{equation}\label{eq:L_operator}
\mathcal{L}^{\delta}(u)(x):=\int_{\R}\frac{g(u(x),u(x+h))-g(u(x-h),u(x))}{h}\omega^{\delta}(h)dh.
\end{equation}
 For convenience we sometimes omit the dependence in $\delta$,   and abbreviate $\mathcal{L}^{\delta}$ as $\mathcal{L}$,  $\omega^{\delta}$ as $\omega$, $u(x,t)$ as
$u(x)$, and  the
partial
derivatives of $g$ as  $g_1, g_1$.
Finally, we set $\Pi_{T}:={\R}\times[0,T]$ and $\Pi_{T}^2:=\Pi_{T} \times \Pi_{T}$.


\subsection{Nonlocal entropy inequality and entropy solutions}

Similarly to the local case, we may define a notion of entropy solution
for the nonlocal models. Moreover, it is shown that an entropy
inequality in the sense of Kru\v{z}kov leads to the uniqueness of the solution
in $L^{\infty}({\R}\times[0,T])\cap C(0,T; L^{1}({\R}))$.

\begin{definition} A function $u$ is an \emph{entropy solution}  of the nonlocal conservation law (\ref{eq:Cauchy}),
 if  $u\in L^{\infty}({\R}\times[0,T])\cap C(0,T; L^{1}({\R}))$ and it satisfies the following Kru\v{z}kov-type entropy inequality:
 \begin{equation}\label{eq:entropy_solution_nonlocal}
 \int_{0}^{T}\int_{{\R}}|u-c|\phi_tdxdt+\int_{0}^{T}\int_{{\R}}\int_{0}^{\delta}\frac{\tau_{h}\phi-\phi}{h}q(u,\tau_{h}u)\omega(h)dh dxdt\geq 0
 \end{equation}
 for every $\phi\in C^{1}_{0}({\R}\times[0,T])$ with $\phi\geq 0$ and any constant $c\in{\R}$. Here
 $\tau_h$ is the shift operator introduced before, $q$ is the nonlocal entropy flux corresponding to the entropy function $\eta(u,c)=|u-c|$, defined as
 \begin{equation}\label{eq:entropy_flux_nonlocal}
 q(a,b;c)=g(a\vee c,b\vee c)-g(a\wedge c,b\wedge c)
 \end{equation}
 or, equivalently, 
 \begin{eqnarray*}
 q(a,b;c)&=&sgn(b-a)\left( \frac{sgn(a-c)+sgn(b-c)}{2} (g(a,b)-g(c,c))  \right.\\
 &&\qquad\qquad
+\left.\frac{sgn(a-c)-sgn(b-c)}{2} (g(c,b)-g(a,c))
\right),
 \end{eqnarray*}
 where we set $sgn(0)=1$. In particular, by the consistency of $g$ in (\ref{eq:g_consistent}), one has 
 \begin{equation*}
 q(u,u;c)=g(u\vee c, u\vee c)-g(u\wedge c, u\wedge c)=f(u\vee c)-f(u\wedge c)=sgn(u-c)[f(u)-f(c)],
\end{equation*}
 which is consistent with the local entropy flux $\mathrm{q}(u,c)=sgn(u-c)(f(u)-f(c))$.
 The inequality (\ref{eq:entropy_solution_nonlocal}) is referred to as the \emph{entropy inequality} for the nonlocal conservation law (\ref{eq:Cauchy}). 
\end{definition}

It is easy to check the following result. 

\begin{lemma} {\label{lemma_q-q}}
Let $q$ be defined in (\ref{eq:entropy_flux_nonlocal}), with $g$ satisfying (\ref{eq:g_Lip}). Then there holds:\\
(i)\quad the Lipschitz continuity of $q$: 
\begin{eqnarray} \label{eq:q_Lip_continuity}
|q(a,b)-q(c,d)|\leq C(|a-c|+|b-d|).
\end{eqnarray}
(ii) \quad  the boundedness of $q$: 
\begin{eqnarray} \label{eq:q_boundedness}
|q(a,b;c)|\leq C\left( |a|+|b|+|c|\right),
\end{eqnarray}
\end{lemma}
 

\subsection{The uniqueness theory}

While the existence of solutions will be established in the next section by introducing a discrete scheme, we can already address the question of the uniqueness of solutions. 

\begin{theorem}{\bf (Uniqueness theory for the nonlocal model)}\label{thm_uniqueness}
 Let $g$ satisfy (\ref{eq:g_Lip}), and  $u$, $v$  be two entropy solutions of the nonlocal conservation law (\ref{eq:Cauchy}) with initial data $u_0$, $v_0\in L^{1}({\R})\cap L^{\infty}({\R})$, respectively. Then, the following contraction property holds for all $t\in(0,T)$,
\begin{eqnarray}\label{eq:L1_constraction}
||u(\cdot,t)-v(\cdot, t)||_{L^{1}({\R})}\leq||u_0-v_0||_{L^1({\R})}. 
\end{eqnarray}
In particular, this implies that $v=u$ almost everywhere in $\Pi_{T}$ whenever the initial data coincide: $u_0=v_0$. 
\end{theorem}

\begin{proof} Recall that $\Pi_T^2=\Pi_T\times\Pi_T$.
 For a nonnegative $\psi=\psi(x,t,y,s)\in C^{\infty}(\Pi_T^2)$ satisfying
$\psi(x,t,y,s)=\psi(y,t,x,s)$, $\psi(x+h,t,y,s)=\psi(x,t,y+h,s)$ and
$ \psi(x,t,y,s)=\psi(x,s,y,t)$,
we define  $$D_x^h \psi(x,t,y,s)= \frac{\psi(x+h,t,y,s)-\psi(x,t,y,s)}{h}$$
and
$$D_y^h \psi(x,t,y,s)= \frac{\psi(x,t,y+h,s)-\psi(x,t,y,s)}{h}.$$

(i) Let $u=u(x,t)$ and $v=v(y,s)$ be  two entropy solutions of the nonlocal conservation law  (\ref{eq:Cauchy}). 
 We consider the entropy  function   $\eta(u,c)=|u-c|$,  and entropy flux $q=q(a,b;c)$  defined  in (\ref{eq:entropy_flux_nonlocal}). 
 
First we take $c=v(y,s)$ in the nonlocal entropy inequality  (\ref{eq:entropy_solution_nonlocal}) for $u$, and integrate over $(y,s)\in\Pi_{T}$:
\begin{eqnarray*}
&&\int_{\Pi_{T}^2}\eta(u(x,t),v(y,s))\partial_t \psi(x,t,y,s)dw\nonumber \\
&&\quad+\int_{\Pi_{T}^2}\int_{0}^{\delta} D_x^h \psi(x,t,y,s)
q(u(x,t),u(x+h,t);v(y,s))\omega(h)dh dw\geq 0,
\end{eqnarray*}
where $dw=dxdtdyds$. Similarly, we take $c=u(x,t)$ in the nonlocal entropy inequality for $v$, integrate over $(x,t)\in\Pi_{T}$, and using the symmetry of $\psi$, we have
\begin{eqnarray*}
&&\int_{\Pi_{T}^2}\eta(v(y,s),u(x,t))\partial_s \psi(x,t,y,s)dw\nonumber \\
&&\quad+\int_{\Pi_{T}^2}\int_{0}^{\delta}  D_y^h \psi(x,t,y,s) q(v(y,s),v(y+h,s);u(x,t))\omega(h)dh dw\geq 0.
\end{eqnarray*}
Adding these two inequalities, and using the properties of $\psi$, we have
\begin{eqnarray}\label{eq:add_entropy_inequalties}
&&\int_{\Pi_{T}^2}\eta(u(x,t),v(y,s))(\partial_t+\partial_s) \psi(x,t,y,s)dw\nonumber \\
&&\quad+\int_{\Pi_{T}^2}\int_{0}^{\delta} D_x^h \psi(x,t,y,s) \left[q\big(u(x,t),u(x+h,t);v(y,s)\big)+\right.
 \nonumber\\
&&\quad\quad\quad\quad \left. +q\big(v(y,s),v(y+h,s);u(x,t)\big)  \right] \omega(h)dh dw\geq 0.\nonumber\\
&&
\end{eqnarray}
Take
\begin{equation*}
\psi(x,t,y,s)=\xi_{\rho}\left(\frac{x-y}{2}\right)\xi_{\rho}\left(\frac{t-s}{2}\right)\phi\left(\frac{x+y}{2},\frac{t+s}{2}\right),
\end{equation*}
where  $\phi=\phi(x,t)\in C^{\infty}_{0}(\Pi_{T}^2 )$ is a non-negative test  function, 
$\xi_{\rho}$ is a scaling of $\xi$:
\begin{equation*}
\xi_{\rho}(x)=\frac{1}{\rho}\xi\left(\frac{x}{\rho}\right),\quad\quad \rho>0,\end{equation*}
and $\xi\in C^{\infty}_{0}(\mathbb{R})$ is a non-negative function satisfying 
\begin{equation*}
\xi(x)=\xi(-x),\quad \xi(x)=0 \mbox{ for } |x|\geq 1, \quad \mbox{ } \int_{\mathbb{R}}\xi(x)dx=1.
\end{equation*}

As $\rho\rightarrow 0$, for the first term in (\ref{eq:add_entropy_inequalties}),
\begin{eqnarray*}
&&\lim_{\rho\rightarrow 0}\int_{\Pi_{T}^2}\eta(u(x,t),v(y,s))(\partial_t+\partial_s) \psi(x,t,y,s)dw =\int_{\Pi_{T}}\eta(u(x,t),v(x,t))\partial_t \phi(x,t)dxdt, 
\end{eqnarray*}
the proof of which can be found in, e.g, proof of Theorem 1 in \cite{karlsen2003uniqueness}. The second term in (\ref{eq:add_entropy_inequalties}) goes to zero:
\begin{eqnarray}\label{eq:flux_term_vanish}
&&\lim_{\rho\rightarrow 0}\int_{\Pi_{T}^2}\int_{0}^{\delta}\frac{\psi(x+h,t,y,s)-\psi(x,t,y,s)}{h}
\Bigg(
q\big(u(x,t),u(x+h,t);v(y,s)\big)
\nonumber\\
&&\quad\quad\quad\quad +q\big(v(y,s),v(y+h,s);u(x,t)\big)  \Bigg) \omega(h)dh dw= 0,\nonumber\\
&&
\end{eqnarray}
 which will be shown later in this proof. 
  Hence sending $\rho\rightarrow 0$,  (\ref{eq:add_entropy_inequalties})  becomes
 \begin{eqnarray*}
\int_{\Pi_{T}}\eta(u(x,t),v(x,t))\partial_t \phi(x,t)dxdt\geq 0, 
\end{eqnarray*}
and this inequality implies the $L^1$-contraction property (\ref{eq:L1_constraction}) (see \cite{karlsen2003uniqueness}). \\
 
 (ii) It now remains to show (\ref{eq:flux_term_vanish}). For the second term in (\ref{eq:add_entropy_inequalties}), 
 we introduce the change of variable
 \begin{eqnarray*}
&&\tilde{x}=\frac{x+y}{2}, \quad z=\frac{x-y}{2}, \quad \tilde{t}=\frac{t+s}{2}, \quad  \tau=\frac{t-s}{2},
\end{eqnarray*}
thus $x=\tilde{x}+z$, $y=\tilde{x}-z$, $t=\tilde{t}+\tau$, and $s=\tilde{t}-\tau$, and 
and use Lebesgue's differentiation  theorem, it becomes
$$
\aligned 
&\lim_{\rho\rightarrow 0}\int_{(\Pi_{T}^2)}\int_{0}^{\delta}\frac{1}{h}
\Bigg(
\xi_{\rho}\left(z+\frac{h}{2}\right)\xi_{\rho}(\tau)\phi\left(\tilde{x}+\frac{h}{2},\tilde{t}\right)-\xi_{\rho}\left(z\right)\xi_{\rho}(\tau)\phi\left(\tilde{x},\tilde{t}\right) \Bigg) 
\\
& \hskip3.cm \Big( 
q\big(u(\tilde{x}+z,\tilde{t}+\tau),u(\tilde{x}+z+h,\tilde{t}+\tau);v(\tilde{x}-z,\tilde{t}-\tau)\big) 
\\
& \quad\quad
 +q\big(v(\tilde{x}-z,\tilde{t}-\tau),v(\tilde{x}-z+h,\tilde{t}-\tau);u(\tilde{x}+z,\tilde{t}+\tau)\big)  \Big)
 \omega(h)dh dw  
\\
&=\int_{\Pi_{T}}\int_{0}^{\delta}\frac{\phi\left(\tilde{x}+\frac{h}{2},\tilde{t}\right)-\phi\left(\tilde{x},\tilde{t}\right)}{h} 
\left[q\big(u(\tilde{x},\tilde{t}),u(\tilde{x}+h,\tilde{t});v(\tilde{x},\tilde{t})\big) \right. 
\\
&\quad\quad\quad\quad  \left. +q\big(v(\tilde{x},\tilde{t}),v(\tilde{x},\tilde{t});u(\tilde{x},\tilde{t})\big) \right] \omega(h)dh d\tilde{x}d\tilde{t}. \quad\quad\quad \label{eq:2nd_term}
\endaligned
$$
Taking $\phi$ as $\phi(x,t)=\chi(t)\varphi_n(x)$ for some $\chi\in C^{\infty}_{0}(0,T)$ and 
\begin{equation*}
 \varphi_n(x)=\int_{\mathbb{R}}\xi(x-y)\mathbf{1}_{|y|<n}dy,\quad\quad n\in(1,\infty), 
\end{equation*}
then each $\varphi_n$ is in $C_{c}^{\infty}(\mathbb{R})$, and  vanishes on $\{x\in\mathbb{R}:||x|-n|>1+\frac{\delta}{2}\}$.  Moreover, we notice that $\lim_{n\rightarrow \infty}\varphi_{n}(x)=\varphi_{\infty}(x)\equiv 1$ and 
\begin{eqnarray}\label{eq:uniform_bded_phi}
\sup_{n}\sup_{x\in\mathbb{R},h\in(0,\delta]} \left| \frac{\varphi_{n}(x+h)-\varphi_{n}(x)}{h}\right| < +\infty.
\end{eqnarray}
Letting $n\rightarrow \infty$ in (\ref{eq:2nd_term}),  by (ii) of  Lemma {\ref{lemma_q-q}}, 
(\ref{eq:uniform_bded_phi}) and (\ref{eq:w_density}) it becomes
\begin{eqnarray*}
\Omega := &&\lim_{n\rightarrow \infty}\left|  \int_{\Pi_{T}}\int_{0}^{\delta}\chi(t)\frac{\varphi_n\left(x+\frac{h}{2}\right)-\varphi_n\left(x\right)}{h} \right.\\
&&\quad\quad\quad\left.  \left[q\big(u(x,t),u(x+h,t);v(x,t)\big)+q\big(v(x,t),v(x,t);u(x,t)\big) \right] \omega(h)dh dxdt        \right|\\
&&\quad \leq C||\chi||_{\infty}\lim_{n\rightarrow \infty} \int_{\Pi_{T}}\int_{0}^{\delta}\big(|u(x,t)|+|v(x,t)|+|u(x+h,t)|+|v(x+h,t)|\big)\\
&&\qquad\qquad\qquad \qquad
\mathbf{1}_{||x|-n|<1+\frac{\delta}{2}}\omega(h)dhdxdt, 
\end{eqnarray*}
thus
\begin{eqnarray*}
\Omega
&&\quad =  C||\chi||_{\infty} \lim_{n\rightarrow \infty} \left\{ \int_{\Pi_{T}}\big(|u(x,t)|+|v(x,t)|\big)
\mathbf{1}_{||x|-n|<1+\frac{3\delta}{2}}dxdt\right.\\
&&\qquad+\left.
 \int_{\Pi_{T}}\int_{0}^{\delta}\big(|u(x+h,t)|+|v(x+h,t)|\big)
\mathbf{1}_{||x|-n|<1+\frac{3\delta}{2}}\omega(h) dh dxdt\right\}\\
&& \quad \leq C||\chi||_{\infty}\lim_{n\rightarrow \infty} \int_{\Pi_{T}}\big(|u(x,t)|+|v(x,t)|\big)
\mathbf{1}_{||x|-n|<1+\frac{5\delta}{2}}dxdt\,=0,\quad\quad
\end{eqnarray*}
by the dominated convergence theorem since $u$ and $v$ belong to $L^{1}(\Pi_T)$.  Thus (\ref{eq:flux_term_vanish})
is established, and the proof is completed.
\end{proof}


\section{A numerical scheme for the nonlocal model}

\subsection{Discretization of the integral term}

We now propose a monotone scheme (\ref{eq:nonlocal_ForwardInTime}) adapted to our nonlocal
conservation law (\ref{eq:Cauchy}).  (Numerical experiments will be reported in \cite{dh16}.) 
Here, we focus on the convergence theory with the aim of establishing the well-posedness of
the nonlocal continuum models and their local limit.

Denote $\Delta x$ and $\Delta t$ as the spatial and time grid-size, $\mathcal{I}_j=\left[\left(j-\frac{1}{2}\right)\Delta x,
\left(j+\frac{1}{2}\right)\Delta x\right)$ and $\mathcal{I}^n=[n\Delta t, (n+1)\Delta t)$ as the spacial and time cells, and 
grid points  $x^n$, $t^n$ as the mid-point  of $\mathcal{I}_j$ and $\mathcal{I}^n$.  Denote $u^n_j$ as the numerical solution at grid point $(x_j, t^n)$.

Following \cite{td13}, at $(x_j,t^n)$, 
we adopt the approximation 
 given  by
\begin{equation}\label{eq:discretizeIntegral}
\int_{0}^{\delta}\frac{g(u(x),u(x+h))-g(u(x-h),u(x))}{h}\omega^{\delta}(h)dh\sim
\sum_{k=1}^{r\vee 1}\left[g_{j,j+k}-g_{j-k,j}\right]W_k, 
\end{equation}
where $r=\left\lfloor\frac{\delta}{\Delta x}\right\rfloor$, $r\vee 1=\max\{r, 1\}$ and 
\begin{eqnarray}\label{eq:Wk}
W_k=\frac{1}{k\Delta x}\int^{k \Delta x}_{(k-1)\Delta x}\omega_\delta(h)dh+\frac{\mathbf{1}_{k=r}}{(r\vee 1) \Delta x}\int_{r\Delta x}^{\delta}\omega_\delta(h)dh
\end{eqnarray}
with $\mathbf{1}_{k=r}$ being the Kronecker-delta function ($1$ for $k=r$ and $0$ otherwise).
In (\ref{eq:Wk}), each $W_k$ depends on both $\Delta x$ and $\delta$.
Since $\omega_\delta$ satisfies (\ref{eq:w_density}), $W_k$ defined above satisfies
\begin{equation}\label{eq:Wk_Sum1}
\Delta x\sum_{k=1}^{r\vee 1} kW_k=1\quad\mbox{ for any pair }\; (\Delta x,\delta).
\end{equation}


\subsection{Numerical scheme}
We consider the following forward-in-time conservative scheme for the nonlocal problem (\ref{eq:Cauchy}):
\begin{eqnarray} \label{eq:nonlocal_ForwardInTime}
\left\{\begin{array}{l l}
& \displaystyle \frac{u^{n+1}_{j}-u^{n}_{j}}{\Delta
t}+\sum_{k=1}^{r\vee 1}\left[g_{j,j+k}-g_{j-k,j}\right]W_k=0,\\[5mm]
& \displaystyle  u^0_j=\frac{1}{\Delta x}\int_{\mathcal{I}_{j}}u_0(x)dx,
\end{array}
\right.
\end{eqnarray}
where $\mathcal{I}_{j}:=[x_{j-\frac{1}{2}},x_{j+\frac{1}{2}})$, and $W_k$ is defined in (\ref{eq:Wk}).
The first equation in (\ref{eq:nonlocal_ForwardInTime}) can be expressed as
\begin{eqnarray}\label{eq:nonlocalschemeH}
&&u^{n+1}_{j}=H(u^n_{j-r},\ldots,u^n_j,\ldots,u^n_{j+r})
\end{eqnarray}
with $H$  defined as
\begin{eqnarray}\label{eq:H}\;\;
H(u_{j-r},\ldots,u_j,\ldots,u_{j+r})=u_{j}-\Delta t\left[
\sum_{k=1}^{r\vee 1}
W_k g_{j,j+k}-  \sum_{k=1}^{r\vee 1}
W_k g_{j-k,j}\right].
\end{eqnarray}
In the following, we will refer scheme (\ref{eq:nonlocal_ForwardInTime}) or  (\ref{eq:nonlocalschemeH}) 
as the {\em nonlocal scheme} for convenience.

If we
fix a spacial grid size $\Delta x$ with $\delta<\Delta x$ and let
$\delta\rightarrow 0$,  the first equation in scheme  (\ref{eq:nonlocal_ForwardInTime}) reduces
to
\begin{equation}
u^{n+1}_{j}=u_j^{n}-\frac{\Delta t}{\Delta
x}\left[g(u_j,u_{j+1})-g(u_{j-1},u_{j})\right],
\end{equation}
which recovers a standard  finite difference scheme for the local
conservation law (\ref{eq:local_ConservationLaw}), where $g$ serves
as the {\em numerical flux} function. In the local case, different choices of numerical flux
lead to different schemes, such as standard Godunove scheme,
linearized Riemann solvers such as  Murmann-Roe scheme, central
schemes such as the Lax-Friedrichs scheme, Rusanov scheme and
Engquist-Osher scheme, etc. It is expected that, by taking $g$ as such
numerical fluxes, we get corresponding nonlocal versions of these
local schemes.


\subsection{Properties of the discrete scheme}
It it important to note in the scheme (\ref{eq:nonlocalschemeH}) the
monotonicity of $H$ which means that  $H$ is nondecreasing with respect to each of its arguments:
\begin{equation}\label{eq:H_monotone}
\frac{\partial H(u^n_{j-r},\ldots,u^n_{j+r})}{\partial
u^n_{i}}\geq 0
\qquad  \text{ for all }  i,j, \mbox{ and } u^n=(\cdots, u^n_{j-1}, u^n_j, u^n_{j+1}, \cdots ).
\end{equation}
In the next lemma we show that our scheme (\ref{eq:nonlocalschemeH}) is monotone under appropriate assumptions on $g$ and a nonlocal CFL condition. \\

\begin{lemma} \label{lemmaMonotone} 
Scheme (\ref{eq:nonlocalschemeH}) satisfies the following:\\
(i) it is  conservative and consistent.\\
(ii) it is monotone if we assume that 
 $g=g(a,b)$ satisfies (\ref{eq:g_consistent}), (\ref{eq:g_monotone}),  (\ref{eq:gi_bded}) on $[B_1,B_2]\times[B_1,B_2]$, where  $B_1=\min_{j}\{u^0_j\}$, $B_2=\max_{j}\{u^0_j\}$, and
the following CFL  condition holds:
\begin{equation}\label{eq:CFL}
\left(\frac{\Delta t}{\Delta
  x}\right)
 \left(\sup_{B_1\leq a,b\leq B_2}|g_1(a,b)|+\sup_{B_1\leq a, b\leq B_2}|g_2(a,b)|\right)\leq 1;
\end{equation}
(iii) 
it has the discrete Maximum Principle under the CFL-condition, (\ref{eq:CFL}):
\begin{equation}\label{eq:MP_discrete}
B_1\leq u^{n}_{j}\leq B_2,\quad\quad \text{ for all } n, j.
\end{equation}
(iv) it satisfies the BV estimate:
\begin{equation}\label{eq:dis-BV}
\sum_i |u^{n}_i - u^{n}_{i-1}| \leq   \sum_i |u^{0}_i - u^{0}_{i-1}|. 
\end{equation}
\end{lemma}

\begin{proof}
(i)\quad  It is straightforward from our definition.\\
(ii)\quad It suffices to show that: $H$ defined in (\ref{eq:H}) is monotone, i.e,  (\ref{eq:H_monotone}) holds.
Since $g$ is monotone  (\ref{eq:g_monotone}),
it is straightforward that for any $k=1,2,\ldots,r$,
\begin{eqnarray*}
\frac{\partial H}{\partial u^n_{j-k}}\geq 0,\quad\quad
\frac{\partial H}{\partial u^n_{j+k}}\geq 0.
\end{eqnarray*}
Moreover, under the assumption  that $B_1\leq u^n_j\leq B_2$ for any $n$ or $j$,
\begin{eqnarray*}\label{eq:H_Uj}
\frac{\partial H}{\partial u^n_{j}}&=&1-(\Delta t
)\sum_{k=1}^{r}[g_1(u_{j},u_{j+k})-g_2(u_{j-k},u_j)]W_k\\
&\geq&1 - \frac{\Delta t}{\Delta x}  \left(\sup_{B_1\leq a,b\leq B_2}|g_1(a,b)|+\sup_{B_1\leq a, b\leq B_2}|g_2(a,b)|\right)
    \left(\Delta x \sum_{k=1}^{r}W_k\right)\\
    &\geq& 1 - \left(\frac{\Delta t}{\Delta
  x}\right)
 \left(\sup_{B_1\leq a,b\leq B_2}|g_1|+\sup_{B_1\leq a, b\leq B_2}|g_2|\right)\geq0.
\end{eqnarray*}
 The last two inequalities comes
 from (\ref{eq:Wk_Sum1}) and the CFL condition
 (\ref{eq:CFL}).   \\
 
The Maximum Principle is a direct consequence of the monotonicity of $H$  (\ref{eq:H_monotone}) and consistency of $g$ (\ref{eq:g_consistent}). Actually, 
 \begin{eqnarray*}
  u^{n+1}_{j}&=&H(u_{j-r},\ldots,u_{j+r})\leq H(\max_j\{u^n_{j}\},\ldots, \max_j\{u^n_{j}\})=\max_j\{u^n_{j}\},\\
    u^{n+1}_{j}&=&H(u_{j-r},\ldots,u_{j+r})\geq H(\min_j\{u^n_{j}\},\ldots, \min_j\{u^n_{j}\})= \min_j\{u^n_{j}\}.\\
   \end{eqnarray*}
That is (\ref{eq:MP_discrete}).
Finally, the estimate (\ref{eq:dis-BV}) is also a direct consequence of the monotonicity property. 
\end{proof}


\section{Convergence of the scheme  and existence theory for the nonlocal model}

\subsection{Total variation estimate}

We now investigate the convergence of  the  numerical
solutions given by the nonlocal scheme (\ref{eq:nonlocal_ForwardInTime}) in two kinds of limiting processes: \\
(1) when the horizon  $\delta$ is fixed and grid-size $\Delta x\rightarrow 0$; \\
(2) when $(\delta,\Delta x)\rightarrow (0,0)$. \\
Throughout, we assume that the ratio  $\frac{\Delta t}{\Delta x}$  is fixed (and satisfies suitable
 CFL condition) as we refine the
 discretization. Thus, as $\Delta x \rightarrow 0$, we have $\Delta t\rightarrow 0$ at the same rate.

 Let $u^n_j$ denote the numerical solution of (\ref{eq:nonlocal_ForwardInTime}),
 $\mathcal{I}_j=\left[\left(j-\frac{1}{2}\right)\Delta x,
\left(j+\frac{1}{2}\right)\Delta x\right)$  and $\mathcal{I}^n=[n\Delta t, (n+1)\Delta t)$. 
Define  piecewise constant function $u^{\Delta,\delta}$ using the grid  function $u^n_j$: 
\begin{equation}\label{eq:u_pw_constant}
u^{\Delta,\delta}(x,t)=\sum_{n=0}^{\infty}\sum_{j=-\infty}^{\infty}u^n_j\mathbf{1}_{\mathcal{I}_j\times\mathcal{I}^{n}}(x,t), 
\end{equation}
where for brevity  $\Delta$ is used to symbolize the dependence on $\Delta t$ and $\Delta x$ (as only
of them is free to change when their ratio is fixed), 
$\mathbf{1}_{\mathcal{I}_j\times\mathcal{I}^{n}}$ is the indicator function which takes value $1$ where $(x,t)\in \mathcal{I}_j\times\mathcal{I}^{n}$, and $0$ otherwise. 
Naturally, $u^{\Delta,\delta}$ depends on the grid size $\Delta x$, $\Delta t$, and the horizon parameter $\delta$. We sometimes write $u^{ \Delta,\delta}$ as $u^{\delta}$  to explicitly emphasize the dependence on $\delta$.

Suppose $\Delta x$ and $\Delta t$ satisfy the CFL condition (\ref{eq:CFL}), then by the discrete Maximum Principle (\ref{eq:MP_discrete}),  we have
\begin{equation}\label{eq:bounded_pw_u}
 ||u^{\Delta,\delta}||_{L^{\infty}(\R\times\R^{+})}\leq ||u_0||_{L^{\infty}(\R)}.
\end{equation}


\begin{lemma}  \label{lemma_BVbound_numericalSolution}
 For a given terminal time $T$,  we consider the nonlocal scheme
(\ref{eq:nonlocal_ForwardInTime}) on $[0,T]$. Assume $u_0\in L^{1}({\R})\cap L^{\infty}({\R})\cap C([0,T];BV(\R))$, $g$ satisfies (\ref{eq:g_consistent}, \ref{eq:g_Lip}, \ref{eq:g_monotone}),  and kernel 
$\omega^{\delta}$ satisfies condition  (\ref{eq:w_density}). 
Also assume that  $\frac{\Delta t}{\Delta x}$ is fixed, and satisfies the CFL condition
(\ref{eq:CFL}).
Then for any $\delta>0$, one has 
\begin{equation}\label{eq:BVbound_numericalSolution}
||u^{\Delta,\delta}(\cdot, t)||_{BV(\mathbf{R})}\leq ||u_0||_{BV(\mathbf{R})},\quad\quad\quad\quad \text{ for all } t\in [0,T].
\end{equation}
\end{lemma}

\begin{proof} 
First we write the 
nonlocal scheme (\ref{eq:nonlocalschemeH}) 
\begin{equation*} 
u^{n+1}_j=H(u^n_{j-r},\ldots,u^n_j,\ldots,u^n_{j+r})
\end{equation*}
as $\vec{u}^{n+1}=\vec{H}(\vec{u}^{n})$. 
By the monotonicity of $H$, we may follow the argument documented in \cite{crandall1980monotone} to concludes:  
\begin{equation*}
||\vec{H}(\vec{u}^n)||_{BV(\mathbf{R})}\leq||\vec{u}^n||_{BV(\mathbf{R})},
\end{equation*}
thus for any $n$,
\begin{equation*}
||\vec{u}^{n+1}||_{BV(\mathbf{R})}=||\vec{H}(\vec{u}^{n})||_{BV(\mathbf{R})}\leq||\vec{u}^{n}||_{BV(\mathbf{R})}\leq\ldots \leq ||\vec{u}_{0}||_{BV(\mathbf{R})}.
\end{equation*}
By definition of $u^{\Delta,\delta}$ as in (\ref{eq:u_pw_constant}),  for any $t$, 
 $$||u^{\Delta,\delta}(\cdot,t)||_{BV(\mathbf{R})}=||\vec{u}^{n+1}||_{BV(\mathbf{R})} \leq ||\vec{u}_{0}||_{BV(\mathbf{R})},$$ 
 so (\ref{eq:BVbound_numericalSolution}) is reached.

\end{proof}

 
 \subsection{Convergence as $(\delta,\Delta x)\rightarrow (\delta,0)$ or $(\delta,\Delta x)\rightarrow (0,0)$}
The main result is given in 
 Theorem \ref{thm_convergence_local_nonlocal_entropy_solution} which shows that,  if we fix $\delta$, and send $\Delta x$ to zero, then the numerical solution converges to the entropy solution of the 
 nonlocal model (\ref{eq:Cauchy}); however,  if we send both $\delta$ and $\Delta x$ to zero, 
 the numerical solution converges to the entropy solution of the limiting local conservation law (\ref{eq:local_ConservationLaw}).  
 But first, we state an result which offers the convergence under stronger assumptions on the initial data.

 \begin{theorem}\label{thm_convergence_local_nonlocal_entropy_solution_BV}{{\bf(Convergence of the numerical solution)}}
 For a given terminal time $T$,  we consider the nonlocal scheme
(\ref{eq:nonlocal_ForwardInTime}) on $[0,T]$. Assume $u_0\in L^{1}({\R})\cap L^{\infty}({\R})\cap C([0,T];BV(\R))$, $g$ satisfies (\ref{eq:g_monotone}, \ref{eq:g_consistent}, \ref{eq:g_Lip}),  and kernel 
$\omega^{\delta}$ satisfies condition  (\ref{eq:w_density}). 
Also assume that  $\frac{\Delta t}{\Delta x}$ is fixed, and satisfies the CFL condition
(\ref{eq:CFL}).
 
(i)\quad For a given fixed $\delta>0$, 
as $\Delta x$ goes to $0$, the solution
$u^{\Delta,\delta}$ converges to the entropy solution of the nonlocal conservation law (\ref{eq:Cauchy} ), $u^{\delta}$, in $L^1_{loc}({\R})$ uniformly for $t\in[0,T]$: 
\begin{equation}\label{eq: converge_L1_nonlocal}
\lim_{\Delta x\rightarrow 0}\sup_{t\in[0,T]}\int_{{\R}}|u^{\Delta,\delta}(\cdot,t)-u^{\delta}(\cdot,t)|dxdt=0.
\end{equation}
(ii)\quad  
As $\delta$ and $\Delta x$ both go to $0$, $u^{\Delta,\delta}$ converges to the entropy solution of  the local conservation law (\ref {eq:local_ConservationLaw}), $u^{local}:[0,T]\rightarrow L^{1}(\R)$,  in $L^1_{loc}({\R})$ uniformly for $t\in[0,T]$: 
\begin{equation}\label{eq: converge_L1_local}
\lim_{(\Delta x,\delta)\rightarrow (0,0)}\sup_{t\in[0,T]}\int_{{\R}}|u^{\Delta,\delta}(\cdot,t)-u^{local}(\cdot,t)|dxdt=0.
\end{equation}
\end{theorem}
 
\begin{proof}
We mainly follows the first part of proof for Theorem 1 in \cite{crandall1980monotone}  by Crandall and
Majda.  Without repeating the same type of calculations, 
we focus on making clear a few facts as follows and highlight the
necessary changes to the original proof to show how the new proof may be constructed. 
Some necessary technical results used in the derivation are shown separately in the later part of the section.

\vskip.15cm 

1) \quad  By Lemma \ref{lemma_mathbbm_g}, our nonlocal scheme (\ref{eq:nonlocal_ForwardInTime})
can be rewritten into the conservative form  (0.4) in \cite{crandall1980monotone}:
\begin{equation}\label{eq:nonlocal_scheme_in _local_form}
u^{n+1}_j=u^n_j-\frac{\Delta t}{\Delta x}[\mathbbm{g}(u^n_{j-r+1},\ldots,u^n_{j+r})-\mathbbm{g}(u^n_{j-r},\ldots,u^n_{j+r-1})],
\end{equation}
with  $\mathbbm{g}$ being Lipschitz continuous and consistent, so the nonlocal scheme (\ref{eq:nonlocal_ForwardInTime}) is conservative and consistent.  (\ref{eq:nonlocal_ForwardInTime}) is also monotone by Lemma \ref{lemmaMonotone}. 

\vskip.15cm 

2) \quad By Lemma \ref{lemma_H(u0)-u0}, we have the nonlocal version of Proposition 3.5 in \cite{crandall1980monotone}. Thus Corollary 3.6 in \cite{crandall1980monotone} also holds for our scheme (\ref{eq:nonlocal_scheme_in _local_form}). 

\vskip.15cm 

3)\quad  Notice that equation (5.2)  in \cite{crandall1980monotone} does not hold for our scheme (\ref{eq:nonlocal_ForwardInTime}), since as $\Delta x$  vanishes, our scheme  may have infinite propagation speed. Indeed, 
at any time $t=K \Delta t$, $K\in \mathbb{N}$,  the value of  $u^0_j$ has influenced the value of $u$ on  $[x_j-D, x_j+D]$,  where 
$D=\delta K$. 
Observing that $\frac{\Delta t}{\Delta x}$ is a constant, say $c$, we have $D=\delta\frac{t}{\Delta t}=\frac{\delta t}{c\Delta x}$, 
so when $\Delta x$ go to zero (no matter $\delta$ is fixed or goes to zero), it is possible that $\frac{\delta}{\Delta x}$ go to $+\infty$, thus $D$ could be unbounded.  

\vskip.15cm 

4) \quad Since equation (5.2)  in \cite{crandall1980monotone} does not hold for our scheme (\ref{eq:nonlocal_ForwardInTime}),  neither does equation (5.3) in \cite{crandall1980monotone}.  That is, fixing a horizon $\delta$,  for our scheme we only have precompactness in $L^{1}_{loc}({\R})$ instead of $L^{1}({\R})$: 
\begin{equation*}
\{u^{\Delta,\delta}(\cdot,t): 0\leq t\leq T, 0\leq \Delta x\leq 1\}  \mbox{ is precompact in }  L^{1}_{loc}({\R}).
\end{equation*}
Therefore,   following the proof therein, we have: 
there exists a subsequence $\{u^{\Delta_k,\delta}\}_k$ (where $(\Delta x)_k\rightarrow 0$ when $k\rightarrow \infty$) and a function $u^{*}\in L^{1}_{loc}({\R})$, such that 
for any compact set $\Omega\subset{\R}$,
 \begin{equation}\label{eq: converge_subsequence_nonlocal}
\lim\limits_{k\rightarrow\infty}\max_{t\in[0,T]}||u^{\Delta_k,\delta}(\cdot,t)-u^{*}(\cdot,t)||_{L^{1} (\Omega)}=0.
\end{equation}
Since $||u^{\Delta,\delta}(\cdot,t)||_{L^1(\R)}\leq||u_0||_{L^{1}(\R)}< +\infty$, $u^{*}\in L^1(\R)$.

\vskip.15cm 

Now, for a fixed $\delta>0$, by (i) of Proposition \ref{proposition_EntropySolutions},  $u^{*}$ is an entropy solution of the nonlocal conservation law corresponding to the given $\delta$, denoted as $u^{\delta}$. The uniqueness of the nonlocal entropy solution (given in Theorem \ref{thm_uniqueness})  guarantees the uniqueness of the limit function 
$u^{*}$, thus 
we have the convergence of the whole sequence $u^{\Delta,\delta}$, reaching  (\ref{eq: converge_L1_nonlocal}).    Also, by Lemma \ref{lemmaMonotone} and Lemma \ref{lemma_BVbound_numericalSolution}, we have
\begin{equation*}
||u^{\Delta,\delta}(\cdot,t)||_{Z}\leq  ||u_0||_{Z} \mbox{  for all } t\in[0,T],\quad\quad\quad   Z=L^{\infty}(\R)\; \mbox{ and }\; BV(\R),   
\end{equation*}
thus, 
$u^{*}=u^{\delta}$ satisfies
\begin{equation}\label{eq:bded_continuity_u*}
||u^{*}(\cdot,t)||_{Z}\leq  ||u_0||_{Z} \mbox{  for all } t\in[0,T].
\end{equation}

6) Similarly, by the precompactness of $\{u^{\Delta,\delta}(\cdot,t): 0\leq t\leq T, 0\leq \Delta x, \delta \leq 1\} $ in $L^{1}_{loc}({\R})$, 
there exists a subsequence  $\{u^{\Delta_k,\delta_k}\}_k$ (where $((\Delta x)_k,\delta_k)\rightarrow(0, 0)$ when $k\rightarrow \infty$)  and a function $u^{**}\in L^{1}({\R})$, such that for any compact set $\Omega\subset{\R}$,
  \begin{equation}\label{eq: converge_subsequence_local}
  \lim\limits_{ k\rightarrow\infty}\max_{t\in[0,T]}||u^{\Delta_k,\delta_k}(\cdot,t)-u^{**}(\cdot,t)||_{L^{1} (\Omega)}=0.\end{equation}
  By (ii) of Proposition \ref{proposition_EntropySolutions},  $u^{**}$ is an entropy solution of the local conservation law, denoted as $u^{local}$. The uniqueness of local entropy solution guarantees the uniqueness of the limit function $u^{**}$, and thus 
we have the convergence of the whole sequence $u^{\Delta,\delta}$,  reaching (\ref{eq: converge_L1_local}). And
\begin{equation}\label{eq:bded_continuity_u**}
||u^{**}(\cdot,t)||_{Z}\leq  ||u_0||_{Z} \mbox{  for all } t\in[0,T].\end{equation}
The proof is completed.
\end{proof}


\subsection{General initial data}

Using similar techniques developed for local conservation laws, the above result leads to the folowin conclusion.

 \begin{theorem}\label{thm_convergence_local_nonlocal_entropy_solution}{{\bf(Main convergence theorem)}}
Assume that $u_0\in L^{1}({\R})\cap L^{\infty}({\R})$, and all other assumptions are the same as in Theorem \ref{thm_convergence_local_nonlocal_entropy_solution_BV} (i.e., we remove only the assumption 
that $u_0\in C([0,T];BV(\R))$).   Then the  conclusion of Theorem \ref{thm_convergence_local_nonlocal_entropy_solution_BV} still holds.
\end{theorem}

\begin{proof}
The proof is the same as the second part of proof  for Theorem 1 in \cite{crandall1980monotone}, and  Theorem \ref{thm_convergence_local_nonlocal_entropy_solution_BV} above can be used to replace the first part of proof for Theorem 1 in \cite{crandall1980monotone}.
\end{proof}

In the remaining part of this section, we sometimes write $u^{*}(x,t)$ and $u^{**}(x,t)$ as
$u^{*}(x)$ and $u^{**}(x)$, especially in Proposition \ref {convergence_local_nonlocal}  and Lemma \ref{lemma_convergence_ae}.. We now present some of the technical results quoted in the proof
of the above theorem \ref{thm_convergence_local_nonlocal_entropy_solution_BV}.

\begin{lemma}\label{lemma_mathbbm_g}
Let $g$ be consistent, monotone and Lipschitz continuous  (see (\ref{eq:g_consistent}), (\ref{eq:g_monotone}),  (\ref{eq:g_Lip})), also let the nonlocal CFL condition (\ref{eq:CFL}) be satisfied. Then the
nonlocal scheme (\ref{eq:nonlocal_ForwardInTime}) can be rewritten as
\begin{equation}\label{eq:NonlocalScheme_localForm}
u^{n+1}_{j}=u^n_j-\frac{\Delta t}{\Delta x} [\mathbbm{g}(u_{j-r+1},\ldots,u_{j+r})-\mathbbm{g}(u_{j-r},\ldots,u_{j+r-1})], 
\end{equation}
where
\begin{equation}\label{eq:mathbbm_g}
\mathbbm{g}(u_{j-r},\ldots,u_{j+r-1})=\sum_{k=1}^{r\vee 1}\sum_{l=1}^{k}g_{j-l,j-l+k}W_k\Delta x, 
\end{equation}
 and $\mathbbm{g}$ is Lipschitz continuous, and consistent with the local flux $f$:  
 \begin{equation}\label{consistency_mathbbm_g}
 \mathbbm{g}(u,\ldots,u)=f(u).
 \end{equation}
\end{lemma}

\begin{proof} Using the definition (\ref{eq:mathbbm_g}) in
scheme (\ref{eq:nonlocal_ForwardInTime}), we get  
$$
\aligned
&\sum_{k=1}^{r\vee 1}(g_{j,j+k}-g_{j-k,j})W_k\Delta x= \sum_{k=1}^{r\vee 1}\sum_{l=1}^{k}  (g_{j+1-l,j+1-l+k}-g_{j-l,j-l+k})W_k\Delta x
\\
& = \sum_{k=1}^{r\vee 1}\sum_{l=1}^{k} g_{j+1-l,j+1-l+k}W_k\Delta x -\sum_{k=1}^{r\vee 1}\sum_{l=1}^{k} g_{j-l,j-l+k}W_k\Delta x
\\
& = \mathbbm{g}(u_{j+1-r},\ldots,u_{j+r})-\mathbbm{g}(u_{j-r},\ldots,u_{j+r-1}).
\endaligned
$$
Thus  (\ref{eq:NonlocalScheme_localForm}) is proven.
Since $g$ is Lipschitz continuous, so is $\mathbbm{g}$.  The consistency of $\mathbbm{g}$ comes from  the consistency of $g$ and normalization condition of $W_k$ given in (\ref{eq:Wk_Sum1}), that is, 
\begin{equation*}
\mathbbm{g}(u,\ldots, u)=\sum_{k=1}^{r\vee 1}\sum_{l=1}^{k}g(u,u)W_k\Delta x=f(u) \sum_{k=1}^{r\vee 1}\sum_{l=1}^{k}W_k\Delta x=f(u).\qedhere
\end{equation*}
\end{proof}

\begin{lemma}\label{lemma_H(u0)-u0}
Assume $u_0\in L^{\infty}(\R)\cap BV(\R)$. For the nonlocal scheme (\ref{eq:nonlocal_ForwardInTime}), we have
\begin{equation}\label{eq:prop3.5}
||H(u)-u||_{L^1(\Delta)}\leq C\Delta t||u^0||_{BV(\Delta)},
\end{equation}
with $H$ is as in (\ref{eq:H}),
and $C$ is independent of $\Delta t$ and $\Delta x$. Here the discrete $L^1$ and BV norms are defined as
\[   ||u||_{L^{1}(\Delta)}=\Delta x\sum_{j=-\infty}^{\infty} |u_j|,\quad\mbox{and}\quad  ||u||_{BV(\Delta)}=\Delta x\sum_{j=-\infty}^{\infty} |u_{j+1}-u_{j}|. \]
\end{lemma}

\begin{proof}
 At any time level $n$,
\begin{eqnarray*}
||H(u)-u||_{L^1(\Delta)}&=&\sum_{j=-\infty}^{\infty} \left|\Delta
t\sum_{k=1}^{r\vee 1}(g_{j,j+k}-g_{j-k,j})W_k\right|\Delta x\\
&\leq&\Delta x\Delta t\sum_{j=-\infty}^{\infty}\sum_{k=1}^{r\vee
1}|g_{j,j+k}-g_{j-k,j}|W_k\\
&\leq& C \Delta x\Delta t  \sum_{j=-\infty}^{\infty}\sum_{k=1}^{r\vee
1}\left(|u_{j}-u_{j+k}|+|u_{j-k}-u_{k}|\right)W_k.
\end{eqnarray*}
Note that,  for initial data $u^0\in L^{1}(\R)$, by Lemma \ref{lemmaMonotone}, scheme (\ref{eq:nonlocalschemeH})
enjoys the Maximum Principle, so $\{u^n_j\}$ is bounded, say by $A$.
Without loss of generality, assume $r\geq 1$.  
We can switch the order of summation, to get
\begin{eqnarray*}
||H(u)-u||_{L^1(\Delta)}&\leq&C \Delta x\Delta t \left[
\sum_{j=-\infty}^{\infty}\sum_{k=1}^{r\vee
1} (|u_{j}-u_{j+k}|  +  
|u_{j-k}-u_{k}|)W_k\right]\\
&\leq& C \Delta x\Delta t  \left[ \sum_{k=1}^{r\vee
1}W_k\sum_{j=-\infty}^{\infty} (|u_{j}-u_{j+k}|+ 
|u_{j-k}-u_{k}|)\right]\\
&\leq& 2 C\Delta x\Delta t 
\sum_{j=-\infty}^{\infty}\sum_{k=1}^{r\vee
1}|u_{j}-u_{j-k}|W_k.
\end{eqnarray*}
Let us break up the terms involving the differences of $\{u_j\}$ into neighboring
differences, that is, 
\begin{eqnarray*}
||H(u)-u||_{L^1(\Delta)}
&\leq&  
C \Delta x\Delta t 
\sum_{k=1}^{r\vee 1} W_k\sum_{l=1}^{k}\sum_{j=-\infty}^{\infty}|u_{j-l+1}-u_{j-l}|\\
&=& C \Delta x\Delta t 
\sum_{k=1}^{r\vee 1} W_k\sum_{l=1}^{k}\sum_{j=-\infty}^{\infty}|u_{j+1}-u_{j}|\\
&\leq& C \Delta x\Delta t 
\sum_{k=1}^{r\vee 1}k W_k\sum_{j=-\infty}^{\infty}|u_{j+1}-u_{j}|\\
&\leq&C\Delta t  \sum_{j=-\infty}^{\infty}|u_{j+1}-u_j|\left( \sum_{k=1}^{r\vee 1} k
W_k \Delta x\right)=C\Delta t ||u||_{BV(\Delta )}, 
\end{eqnarray*}
where the constant $C$ only depends on $g$, and the last inequality comes from (\ref{eq:Wk_Sum1}).
\end{proof}

\begin{remark}
 Lemma \ref{lemma_H(u0)-u0} mimics Proposition 3.5 in \cite{crandall1980monotone}. In Proposition 3.5 in \cite{crandall1980monotone},   the constant $C$ on the 
 right hand side of (\ref{eq:prop3.5}) depends on $r$,  the number of cells involved in numerical flux $\mathbbm{g}$.
 However, in Lemma \ref{lemma_H(u0)-u0}, we are able to bound the left hand side of (\ref{eq:prop3.5})
in a way such that the coefficient $C$ is independent of $r$, and independent of $\Delta x$, $\Delta t$  and $\delta$.  
\end{remark}

\begin{proposition}\label{proposition_EntropySolutions}
Let $u_0\in BV(\R)$.  
Suppose that all the assumptions of Theorem \ref{thm_convergence_local_nonlocal_entropy_solution} hold, and $u^{\Delta,\delta}$ is defined in (\ref{eq:u_pw_constant}).\\
(i) 
$u^{*}$ in (\ref{eq: converge_subsequence_nonlocal})  is an entropy solution of the nonlocal conservation law.\\
(ii)  
 $u^{**}$ in (\ref{eq: converge_subsequence_local})   is an entropy solution of the local conservation law.
\end{proposition}

\begin{proof}
By the CFL condition (\ref{eq:CFL}) and Lemma \ref{lemmaMonotone},   the function $H$, given in scheme (\ref{eq:nonlocal_ForwardInTime}) via
$u^{n+1}_{j}=H(u^n_{j-r},\ldots, u^n_{j+r})
$
is nondeceasing with respect to each of its arguments.  So for any constant $c\in{\R}$, 
\begin{eqnarray*}
H(u^n_{j-r}\wedge c,\ldots, u^n_{j+r}\wedge c)\leq & u^{n+1}_j &\leq H(u^n_{j-r}\vee c,\ldots, u^n_{j+r}\vee c),\\
H(u^n_{j-r}\wedge c,\ldots, u^n_{j+r}\wedge c) \leq & c&\leq H(u^n_{j-r}\vee  c,\ldots, u^n_{j+r}\vee  c),\\
u^{n+1}_j \vee c&\leq& H(u^n_{j-r}\vee c,\ldots, u^n_{j+r}\vee c),\\
u^{n+1}_j \wedge c&\geq& H(u^n_{j-r}\wedge c,\ldots, u^n_{j+r}\wedge c).
\end{eqnarray*}
Subtracting the last two inequalities,
\begin{eqnarray*}
u^{n+1}_j \vee c-u^{n+1}_j \wedge c&\leq& H(u^n_{j-r}\vee c,\ldots, u^n_{j+r}\vee c)-H(u^n_{j-r}\wedge c,\ldots, u^n_{j+r}\wedge c),
\end{eqnarray*}
thus we find 
\begin{eqnarray*}
u^{n+1}_j \vee c-u^{n+1}_j \wedge c &\leq&  
-\Delta t \sum_{k=1}^{r\vee 1}\big[ g(u^n_j\vee c, u^n_{j+k}\vee c)- g(u^n_{j-k}\vee c, u^n_{j}\vee c)\big]W_k\\
&& 
+\Delta t \sum_{k=1}^{r\wedge 1}\big[ g(u^n_j\wedge c, u^n_{j+k}\wedge c)- g(u^n_{j-k}\wedge c, u^n_{j}\wedge c)\big]W_k.
\end{eqnarray*}
and
\begin{eqnarray*}
|u^{n+1}_j-c|&\leq&|u^n_j-c| -\Delta t \sum_{k=1}^{r\vee 1}\big[ g(u^n_j\vee c, u^n_{j+k}\vee c)- g(u^n_j\wedge c, u^n_{j+k}\wedge c)\big]W_k\\
&&+\Delta t \sum_{k=1}^{r\vee 1}\big[ g(u^n_{j-k}\vee c, u^n_{j}\vee c)- g(u^n_{j-k}\wedge c, u^n_{j}\wedge c)\big]W_k
\end{eqnarray*}
\begin{eqnarray*}
\frac{|u^{n+1}_j-c|-|u^{n}_j-c|}{\Delta t}&\leq&
-\sum_{k=1}^{r\vee 1}\big[q(u^n_{j},u^n_{j+k};c)-q(u^n_{j-k},u^n_{j};c)\big]W_k,
\end{eqnarray*}
where $q$ is the nonlocal entropy flux defined in (\ref{eq:entropy_flux_nonlocal}).
We use the notaiton 
\begin{equation}\label{eq:phi_j}
\phi\in C_{c}^{\infty}(\mathbb{R^{+}}\times{\R}),\quad \phi\geq0,\quad 
\phi^n_j:=\frac{\phi(x_{j-\frac{1}{2}})-\phi(x_{j+\frac{1}{2})}}{\Delta x}.
\end{equation}
Multiplying the above inequality by $\Delta t\Delta x\phi^n_j\geq 0$, and summing with respect to $j$ and $n$, we get
\begin{eqnarray}
&&\Delta t\Delta x \sum_{n=0}^{N}\sum_{j=-\infty}^{\infty}\frac{|u^{n+1}_j-c|-|u^{n}_j-c|}{\Delta t}\phi^n_j 
\nonumber\\
&&\quad + \Delta t\Delta x\sum_{n=0}^{N}\sum_{j=-\infty}^{\infty}\phi^n_j\sum_{k=1}^{r\vee 1}\big[q(u^n_{j},u^n_{j+k};c)-q(u^n_{j-k},u^n_{j};c)\big]W_k\leq 0.\label{discrete_2ndTerm_EntropyInequality}
\end{eqnarray}
Let $\{(\Delta x)_k\}_k$ be the subsequence of $\Delta x$ in (\ref{eq: converge_subsequence_nonlocal}). 
By (\ref{eq: converge_subsequence_nonlocal}),  and with the same argument as in \cite{crandall1980monotone}, 
the first term in the above  (\ref{discrete_2ndTerm_EntropyInequality})  converges as follows
\begin{eqnarray*}
&&\lim_{(\Delta x)_k\rightarrow 0}\Delta t\Delta x \sum_{n=0}^{N}\sum_{j=-\infty}^{\infty}\frac{|u^{n+1}_j-c|-|u^{n}_j-c|}{\Delta t}\phi^n_j= -\int_{0}^{T}\int_{{\R}}|u^{*}-c|\phi_t dxdt.
\end{eqnarray*}
Let $\{\delta_k, (\Delta x)_k\}_k$ be the subsequence of $\{\delta, \Delta x\}$ in (\ref{eq: converge_subsequence_local}). By  (\ref{eq: converge_subsequence_local}), a similar argument yields
\begin{eqnarray*}
&&\lim_{\delta_k\rightarrow 0,(\Delta x)_k\rightarrow 0}\Delta t\Delta x \sum_{n=0}^{N}\sum_{j=-\infty}^{\infty}\frac{|u^{n+1}_j-c|-|u^{n}_j-c|}{\Delta t}\phi^n_j= \int_{0}^{T}\int_{{\R}}|u^{**}-c|\phi_t dxdt.
\end{eqnarray*}
The proof can be completed by applying  the Proposition  \ref{convergence_local_nonlocal} 
proved below to the  other term in  (\ref{discrete_2ndTerm_EntropyInequality}).
\end{proof}

\begin{proposition}\label{convergence_local_nonlocal} 
Assume that $u_0\in BV(\R)$ and $\phi$ satisfies (\ref{eq:phi_j})
and all the assumptions of Proposition \ref{proposition_EntropySolutions} hold. 
For a given $c\in{\R}$, and denote the nonlocal entropy flux  $q(a,b;c)$ as $q(a,b)$  for brevity. \\
 \noindent(i)\quad For a fixed $\delta>0$, let $\{(\Delta x)_l\}_l$ be the subsequence of $\Delta x$ in 
 (\ref{eq: converge_subsequence_nonlocal}).  Then one has
\begin{eqnarray}
&&\lim_{(\Delta x)_l\rightarrow 0}\Delta t\Delta x\sum_{n=0}^{N}\sum_{j=-\infty}^{\infty}\sum_{k=1}^{r\vee 1}\phi^n_j[q(u^n_j,u^n_{j+k})-q(u^n_{j-k},u^n_{j})]W_k \nonumber\\
&&\quad = 
-\int_{0}^{T}\int_{{\R}}\int_{0}^{\delta}[\phi(x)-\phi(x+h)] q(u^{*}(x),u^{*}(x+h)) \frac{\omega(h)}{h}dhdxdt.\label{eq:convegence_NonocalEntropyIntegral}
\end{eqnarray}
 \noindent(ii)\quad  
 Let $\{\delta_l, (\Delta x)_l\}_l$ be the subsequence of $\{\delta, \Delta x\}$ in (\ref{eq: converge_subsequence_local}).  Then: 
\begin{eqnarray}
&&\lim_{\delta_l\rightarrow 0,(\Delta x)_l\rightarrow 0}\Delta t\Delta x \sum_{n=0}^{N}\sum_{j=-\infty}^{\infty}\sum_{k=1}^{r\vee 1}\phi^n_j[q(u^n_j,u^n_{j+k})-q(u^n_{j-k},u^n_{j})]W_k\nonumber \\
&&\quad =
-\int_{0}^{T}\int_{{\R}}\phi_x sgn(u^{**}-c) (f(u^{**})-f(c))dxdt.\label{eq:convegence_LocalEntropyIntegral}
\end{eqnarray}
\end{proposition}

\begin{proof}
  Using summation by parts, we have 
 \begin{eqnarray*}
 && \Delta t\Delta x\sum_{n=0}^{N}\sum_{j=-\infty}^{\infty}\phi^n_j\sum_{k=1}^{r\vee 1}\big[q(u^n_{j},u^n_{j+k})-q(u^n_{j-k},u^n_{j})\big]W_k\nonumber\\
  &&\quad=-\Delta t\Delta x\sum_{n=0}^{N}\sum_{j=-\infty}^{\infty}\sum_{k=1}^{r\vee 1}(\phi_{j+k}-\phi_{j})q(u^n_{j},u^n_{j+k})W_k \nonumber\\
    &&\quad= -\Delta t\Delta x\sum_{n=0}^{N}\sum_{j=-\infty}^{\infty}\sum_{k=1}^{r\vee 1}\left(\frac{\phi_{j+k}-\phi_{j}}{k\Delta x}\right) q(u^n_{j},u^n_{j+k}) (kW_k) \Delta x) \nonumber\\
 &&\quad= -\Delta t (\Delta x )^2\sum_{n=0}^{N}\sum_{j=-\infty}^{\infty}\sum_{k=1}^{r}\frac{\phi_{j}-\phi_{j+k}}{k\Delta x}q_{j,j+k}\left(\frac{1}{\Delta x}\int_{(k-1)\Delta x}^{k\Delta x}\omega(s)ds\right)
 \end{eqnarray*}
 It can be further written as 
 \begin{align}\label{eq: form_unify}
 &  \Delta t\Delta x\sum_{n=0}^{N}\sum_{j=-\infty}^{\infty}\phi^n_j\sum_{k=1}^{r\vee 1}\big[q(u^n_{j},u^n_{j+k})-q(u^n_{j-k},u^n_{j})\big]W_k 
=
 -\int_{0}^{T}\int_{\R}\int_{0}^{\delta}
G(\Delta,x,t,h)
dh dx dt,  \nonumber
 \end{align}
 \noindent where 
 $$G(\Delta,x,t,h)
= \frac{\phi^{\Delta}(x,t)-\phi^{\Delta}(x+h^{\Delta}(h),t)}{h^{\Delta}(h)}q(u^{\Delta}(x),u^{\Delta}(x+h^{\Delta}(h)))\bar{\omega}^{\Delta}(h),$$
and
  $u^{\Delta}$ is defined in (\ref{eq:u_pw_constant}),
\begin{eqnarray*}
\phi^{\Delta}(x,t)&=&\sum_{j=-\infty}^{\infty}\sum_{n=0}^{\infty}\left( \frac{\phi(x_{j+\frac{1}{2}})-\phi(x_{j-\frac{1}{2}})}{\Delta x}(x-x_{j-\frac{1}{2}})+\phi(x_{j-\frac{1}{2}}) \right)
\mathbf{1}_{\mathcal{I}_j\times\mathcal{I}^{n}}(x,t),\\
\bar{\omega}_{k}&:=&\frac{1}{\Delta x}\int_{\hat{\mathcal{I}}_k}\omega(s)ds,\ \quad \mbox{for } \; 
\hat{\mathcal{I}}_k=[(k-1)\Delta x,k\Delta x),
\\
\bar{\omega}^{\Delta}(h)&=&\sum_{k=1}^{r}\bar{\omega}_{k} \mathbf{1}_{\hat{\mathcal{I}}_k}(h)=\sum_{k=1}^{r}\left(\frac{1}{\Delta x}\int_{\hat{\mathcal{I}}_k}\omega(s)ds\right) \mathbf{1}_{\hat{\mathcal{I}}_k}(h),\\
h^{\Delta}(h)&=&\sum_{k=1}^{r}(k\Delta x)\mathbf{1}_{\hat{\mathcal{I}}_k}(h).
\end{eqnarray*}
Note that these functions are piecewise constant except for $\phi^{\Delta}(x,t)$ which is piecewise linear  about $x$ on the interval $[x_{j-\frac{1}{2}},x_{j+\frac{1}{2}}]$. 
 We also take note of the following facts:\\
  \begin{equation} \label{eq: approx_a}  \lim\limits_{\Delta x\rightarrow 0}h^{\Delta}(h)=h, \quad\quad h\in(0,\delta] \end{equation}
   \begin{equation} \label{eq: approx_b} 
\lim_{\Delta x\rightarrow 0} \bar{\omega}^{\Delta}(h)=\omega(h),\quad\ \text{ for all } h\in(0,\delta];  \quad\quad 
 \end{equation}
  \begin{equation} \label{eq: approx_b2} 
 ||\bar{\omega}^{\Delta}||_{L^{1}(0,\delta)}=\int_{0}^{\delta}\bar{\omega}^{\Delta}(h)dh
=\sum_{k=1}^{r}\int_{(k-1)\Delta x}^{k\Delta x}\omega(s)ds\equiv 1,\quad \text{ for all } \delta>0.
 \end{equation}
 Meanwhile, since $ \phi^{\Delta}(x,t)\rightarrow \phi(x,t)$ uniformly, $\phi$ is continuous, and 
 $h^{\Delta}(h)\rightarrow h$, 
 \begin{equation} \label{eq: approx_c} 
 \lim_{\Delta x\rightarrow 0} \phi^{\Delta}(x+h^{\Delta}(h),t)\rightarrow \phi(x+h,t) \quad \mbox{ a.e } (x,t).
 \end{equation}
In addition, we have 
 \begin{eqnarray}  \label{eq: approx_d} 
\frac{\phi^{\Delta}(x,t)-\phi^{\Delta}(x+h^{\Delta}(h),t)}{h^{\Delta}(h)}\leq C||\phi_{x}||_{L^{\infty}({\R}\times{\R}^{+})}.
\end{eqnarray}

Furthermore, when (\ref{eq: converge_subsequence_nonlocal}) holds, apparently 
 \begin{equation}\label{eq: approx_e1} 
 \lim\limits_{k \rightarrow \infty}u^{\Delta_k,\delta}(x,t)=u^{*}(x,t),\quad\quad \mbox{  for any } t\in[0,T]. 
    \end{equation}
Moreover, by (\ref{eq:bded_continuity_u*}),  $u^{*}(\cdot, t)\in BV(\R)$,  so $u^{*}$ has at most countably many discontinuities and is continuous almost everywhere on $\R$.  
Hence for almost everywhere $(x,t)\in[a-\delta,b+\delta]\times[0,T]$, by (\ref{eq: approx_a}),
\begin{equation}\label{eq: approx_e2} 
\lim_{\Delta x\rightarrow 0}u^{*}(x+h^{\Delta }(h),t)=u^{*}(x+ \lim_{\Delta x\rightarrow 0}h^{\Delta }(h),t) = u^{*}(x+h,t), 
\end{equation}
These facts will be useful in the proof of this proposition and Lemma \ref{lemma_convergence_ae}.

(i)  We start by proving  (\ref{eq:convegence_NonocalEntropyIntegral}). For convenience of notation, we just denote $(\Delta x)_l$ as $\Delta x$, but we should always keep in mind:  these $\{(\Delta x)_l\}$ corresponds to the subsequence of $u^{\Delta}$ that converges to $u^{*}$ as in (\ref{eq: converge_subsequence_nonlocal}). \\ 
It suffices to show that,  when $\delta$ is fixed and $\Delta x\rightarrow 0$,  we have the convergence 
\begin{eqnarray*}
&&\lim_{\Delta x\rightarrow 0}\int_{0}^{T}\int_{\R}\int_{0}^{\delta}
G(\Delta,x,t,h) 
dh dx dt = \int_{0}^{T}\int_{{\R}}\int_{0}^{\delta}\frac{\phi(x)-\phi(x+h)}{h}q(u(x),u(x+h)) \omega^\delta(h)dhdxdt.
\end{eqnarray*}  

Actually, applying the dominated convergence theorem (conditions will be checked shortly), we have
\begin{eqnarray*}
&&\lim_{\Delta x\rightarrow 0}\int_{0}^{\delta}\left(\int_{{\R}^{+}\times\R} 
G(\Delta,x,t,h)
 dx dt\right)dh
=\int_{0}^{\delta}\lim_{\Delta x\rightarrow 0}\left(\int_{{\R}^{+}\times\R} 
G(\Delta,x,t,h)
 dx dt\right)dh\\
 &&\quad= 
 \int_{0}^{\delta}\int_{{\R}^{+}\times{\R}}[\phi(x,t)-\phi(x+h,t)]q(u(x,t),u(x+h,t))\frac{\omega^{\delta}(h)}{h}dx dtdh.
 \end{eqnarray*}

It remains to check two conditions for dominated convergence theorem: \\
\emph{Condition 1:} For a.e $h\in(0,\delta]$, 
\begin{eqnarray*} 
&&\lim_{\Delta x\rightarrow 0}\int_{0}^{T}\int_{\R} 
G(\Delta,x,t,h) dx dt
=\int_{0}^{T}\int_{\R}\frac{\phi(x,t)-\phi(x+h,t)}{h}q(u^{*}(x,t),u^{*}(x+h,t)) \omega^{\delta}(h) dx dt.\quad\quad
 \end{eqnarray*}
 This is based on Lemma \ref{lemma_convergence_ae} and (\ref{eq: approx_b}).

\noindent \emph{Condition 2:}  there exists a function $Y\in L^{1}(0,\delta)$, such that
$
\left|\int_{0}^{T}\int_{\R} 
G(\Delta,x,t,h)
 dx dt\right|\leq Y(h).
$

Actually, we have 
\begin{eqnarray*}
\left|\int_{0}^{T}\int_{a-\delta}^{b+\delta} 
 G(\Delta,x,t,h)
 dx dt\right|
 && \leq\bar{\omega}^{\Delta}(h)\cdot C\cdot\int_{0}^{T}\int_{a-\delta}^{b+\delta}\frac{\phi^{\Delta}(x,t)-\phi^{\Delta}(x+h^{\Delta}(h),t)}{h^{\Delta}(h)} dx dt\\
 &&
 \leq \bar{\omega}^{\Delta}(h)\cdot C\cdot T(b-a+2\delta)||\phi_{x}||_{L^{\infty}({\R}\times{\R}^{+})}=C_{\phi,u_0, g,\delta} \bar{\omega}^{\Delta}(h),
\end{eqnarray*}

\noindent  by the boundedness on $\phi$ (see (\ref{eq: approx_c})-(\ref{eq: approx_d})), the boundeness of $q$ (see (\ref {eq:q_boundedness})) and the boundedness of $u^{\Delta}$ (see (\ref{eq:bounded_pw_u})).  Noting that $\bar{\omega}^{\Delta}$ is integrable (\ref{eq: approx_b2}), we can take 
$
Y(h)=C_{\phi, u_0, g, \delta} \bar{\omega}^{\Delta}(h).
$

(ii) For convenience of notation, we just denote  $\delta_l$ as $\delta$, and $(\Delta x)_l$ as $\Delta x$, but we should always keep in mind:  these $\{\delta_l, (\Delta x)_l\}$ corresponds to the subsequence of $\{u^{\Delta,\delta}\}$ that converges to $u^{**}$ as in (\ref{eq: converge_subsequence_local}). \\  

\indent To establish  (\ref{eq:convegence_LocalEntropyIntegral}), it suffices to show that, when $\delta$  and $\Delta x$ both go to zero,  we have the convergence property
$$
\aligned
\lim_{(\delta,\Delta x)\rightarrow (0,0)}\int_{0}^{T}\int_{\R}\int_{0}^{\delta} 
G(\Delta,x,t,h)
dh dx dt
& = \int_{0}^{\infty}\int_{\R} \phi_x(x,t) q(u^{**}(x,t),u^{**}(x,t)) dx dt
\\
& = \int_{0}^{T}\int_{\R} \phi_x(x,t)sgn(u^{**}-c) (f(u^{**})-f(c)) dx dt.
\endaligned
$$
 Actually, for any $\varepsilon>0$,  when $\delta$ and $\Delta x$ are small enough, 
\begin{eqnarray*}
&&\left| \frac{\phi^{\Delta}(x,t)-\phi^{\Delta}(x+h^{\Delta}(h),t)}{h^{\Delta}(h)}q(u^{\Delta}(x),u^{\Delta}(x+h^{\Delta }(h)))-\phi_x(x)q(u^{**}(x),u^{**}(x)) \right|\\
&&\quad \leq 
\left|\frac{\phi^{\Delta}(x+h^{\Delta}(h))-\phi^{\Delta}(x)}{h^{\Delta}(h)}-\frac{\phi(x+h)-\phi(x)}{h}\right|
\left|q(u^{\Delta}(x),u^{\Delta}(x+h^{\Delta}))\right| \\
&&\qquad +\left|\frac{\phi(x+h)-\phi(x)}{h}-\phi_x(x)\right|
\left|q(u^{\Delta}(x),u^{\Delta}(x+h^{\Delta}))\right| \\
&&\qquad 
+ |\phi_x(x)| \left| q(u^{\Delta}(x),u^{\Delta}(x+h^{\Delta}(h)))-q(u^{**}(x),u^{**}(x))  \right|\\
&&\quad \leq 2\varepsilon C_{\phi,g}(|u^{\Delta}(x)|+|u^{\Delta}(x+h^{\Delta}(h))|+1)+C_{\phi,g }\left(|u^{\Delta}(x)-u^{**}(x)|\right.\\
&& \qquad\quad\left. +|u^{\Delta}(x+h^{\Delta})-u^{**}(x)|\right)\\
&&\quad \leq C_{\phi,g,u_0}\varepsilon+C_{\phi,g}\left(|u^{\Delta}(x)-u^{**}(x)|+|u^{\Delta}(x+h^{\Delta})-u^{**}(x+h^{\Delta})|\right.\\
&&\qquad\quad \left.+|u^{**}(x+h^{\Delta})-u^{**}(x)| \right)\\
&&\quad \leq C_{\phi,g,u_0}\varepsilon+C_{\phi,g}|u^{\Delta}(x+h^{\Delta})-u^{**}(x+h^{\Delta})|, \quad\quad
\mbox{ a.e } (x,t)\in\R\times\R^{+}.
\end{eqnarray*}
The  inequalities above are based on the boundedess of $\phi$ (\ref{eq: approx_c}-\ref{eq: approx_d}), the boundedness and Lipschitz continuity of $q$ (\ref{eq:q_Lip_continuity}, \ref{eq:q_boundedness}), and the convergence of $u^{\Delta, \delta}$ to $u^{**}$ as $\Delta x\rightarrow 0$ given in (\ref{eq: converge_subsequence_local}). 

Let the compact support of $\phi$ be in an interval $[a,b]\times[0,T]$. Using the above inequality,  and noting that $||\bar{\omega}^{\delta}||_{L^{1}(\R)}=1$  by (\ref{eq: approx_b}),  we have
{\footnotesize
\begin{eqnarray*}
\Omega := 
&&\left|\int_{0}^{T}\int_{a-\delta}^{b+\delta}\int_{0}^{\delta}
G(\Delta,x,t,h) 
dh dx dt
- \int_{0}^{T}\int_{a-\delta}^{b+\delta}\phi_x(x)q(u^{**}(x),u^{**}(x))  dx dt\right|\\
&&\quad \leq \int_{0}^{T}\int_{a-\delta}^{b+\delta}\int_{0}^{\delta}\left| 
G(\Delta,x,t,h) - \phi_x(x)q(u^{**}(x),u^{**}(x))
\right| 
dh dx dt\\
&&\quad \leq \int_{0}^{T}\int_{a-\delta}^{b+\delta}\int_{0}^{\delta}  \left(C_{\phi,g,u_0} \varepsilon +C_{\phi,g}|u^{\Delta}(x+h^{\Delta})-u^{**}(x+h^{\Delta})|\right)\bar{\omega}^{\Delta}(h)
dh dx dt, 
\end{eqnarray*}
}
thus 
\begin{eqnarray*}
\Omega \leq 
&&  C_{\phi,g,u_0} T (b-a+2\delta)\varepsilon +\int_{0}^{T}\int_{a-\delta}^{b+\delta}\int_{0}^{\delta}|u^{\Delta}(x+h^{\Delta})-u^{**}(x+h^{\Delta})|\bar{\omega}^{\Delta}(h)
dh dx dt\\
&&  \leq C_{\phi,g,u_0}\varepsilon+\int_{0}^{T}\int_{0}^{\delta}\int_{a-\delta+h^{\Delta}(h)}^{b+\delta+h^{\Delta}(h)}|u^{\Delta}(y)-u^{**}(y)|\bar{\omega}^{\Delta}(h) dy dh dt\\
&& \leq C_{\phi,g,u_0}\varepsilon+\int_{0}^{T}\int_{a-\delta}^{b+2\delta}|u^{\Delta}(y)-u^{**}(y)|dydt \left(\int_{0}^{\delta}\bar{\omega}^{\Delta}(h) dh\right)\\
&& = C_{\phi,g,u_0}\varepsilon+\int_{0}^{T}\int_{a-\delta}^{b+2\delta}|u^{\Delta}(x)-u^{**}(x)|dxdt.
\end{eqnarray*}

By (\ref{eq: converge_subsequence_local}), when $\Delta x$ is small enough, 
one has 
$\int_{0}^{T}\int_{a-\delta}^{b+2\delta}|u^{\Delta}(x)-u^{**}(x)|dxdt<\varepsilon$, 
and thus
{\footnotesize
\begin{eqnarray*}
&&\int_{0}^{T}\int_{a-\delta}^{b+\delta}\int_{0}^{\delta}
\int_{0}^{T}\int_{a-\delta}^{b+\delta}\int_{0}^{\delta}\left| 
G(\Delta,x,t,h) - \phi_x(x)q(u^{**}(x),u^{**}(x))
\right|
dh dx dt
\leq C_{\phi,g,u_0}\varepsilon,
\end{eqnarray*}
}
Now (\ref{eq:convegence_LocalEntropyIntegral}) is obtained, and the proof is completed.
\end{proof}

\begin{lemma}\label{lemma_convergence_ae}
Suppose all the assumptions of Proposition \ref{convergence_local_nonlocal}  hold.   
Let  $u^{(\Delta x)_l}$ be the subsequence that converges to $u^{*}$ as in (\ref{eq: converge_subsequence_nonlocal}), and for convenience we denote  $\{(\Delta x)_l\}$ as $\Delta x$. \\
Then for  a.e  $h\in(0,\delta]$, 
 \begin{eqnarray}\label{eq:term_convergence}
&&\lim_{\Delta x\rightarrow 0} \int_{0}^{T}\int_{\R}  
\frac{\phi^{\Delta}(x,t)-\phi^{\Delta}(x+h^{\Delta}(h),t)}{h^{\Delta}(h)}q(u^{\Delta}(x),u^{\Delta}(x+h^{\Delta }(h)))
 dx dt\quad\quad\quad\nonumber\\
&&\qquad\quad=
\int_{0}^{T}\int_{\R} \frac{\phi(x,t)-\phi(x+h,t)}{h}q(u^{*}(x,t),u^{*}(x+h,t))dx dt. 
 \end{eqnarray}
\end{lemma}

\begin{proof}
 Denote the support of $\phi(x,t)$ is in $[a,b]\times[0,T]$,  then (\ref{eq:term_convergence}) is equivalent to
\begin{eqnarray*}
&&\lim_{\Delta x\rightarrow 0} \int_{0}^{T}\int_{a-\delta}^{b+\delta} 
\frac{\phi^{\Delta}(x,t)-\phi^{\Delta}(x+h^{\Delta}(h),t)}{h^{\Delta}(h)}q(u^{\Delta}(x),u^{\Delta}(x+h^{\Delta }(h)))
 dx dt\quad\quad\quad\nonumber\\
&&\quad= \int_{0}^{T}\int_{a-\delta}^{b+\delta} \frac{\phi(x,t)-\phi(x+h,t)}{h}q(u^{*}(x,t),u^{*}(x+h,t))dx dt. 
\end{eqnarray*}
By the fact that
$$
\left.
\begin{array}{ll} \displaystyle
f_{n}\rightarrow f  \mbox{ a.e, } & \sup_{n}|f_n|\leq C\mathbf{1}_{E}, \quad \sup_{x}|f|< +\infty, \\
\displaystyle g_{n}\rightarrow g \in L^{1},   &  \sup_{n,x}g< +\infty  \quad\quad\quad\quad\quad\quad\quad\quad\quad
\end{array}
\right\}\Rightarrow\quad  f_{n}g_{n}\rightarrow fg\quad  \mbox{   in } L^{1}, 
$$
we only need to show the following two conditions: for a.e. $h\in(0,\delta]$, we have 
(i):
 \begin{equation*}
 \left\{
\begin{array}{ll}
 \displaystyle \lim_{\Delta x\rightarrow 0}\frac{\phi^{\Delta}(x,t)-\phi^{\Delta}(x+h^{\Delta}(h),t)}{h^{\Delta}(h)}=\frac{\phi(x,t)-\phi(x+h,t)}{h}  \quad \mbox{   a.e   } (x,t)\in{\R}\times{\R}^{+}
 \\[0.5cm]
\displaystyle  \sup_{\Delta x} \left|\frac{\phi^{\Delta}(x,t)-\phi^{\Delta}(x+h^{\Delta}(h),t)}{h^{\Delta}(h)}\right|< C\mathbf{1}_{E}, \quad  \sup_{x,t}\left|\frac{\phi(x,t)-\phi(x+h,t)}{h}\right|< +\infty,
\end{array}
\right.
 \end{equation*}
 where $E\subset \R$ is a bounded set; and  (ii): 
  \begin{eqnarray*}
  \left\{
  \begin{array}{l}\displaystyle \lim_{\Delta x\rightarrow 0}\int_{0}^{T}\int_{a-\delta}^{b+\delta}|q(u^{\Delta}(x),u^{\Delta}(x+h^{\Delta }(h)))-q(u^{*}(x),u^{*}(x+h))|
 dx dt=0,
\\[0.5cm]
\displaystyle \sup_{\Delta x, x,t} |q(u^{\Delta}(x),u^{\Delta}(x+h^{\Delta }(h)))|<+\infty.
\end{array}
\right.
 \end{eqnarray*}
 
We note first that (i) is obvious by (\ref{eq: approx_c}-\ref{eq: approx_d}). Next we show (ii). The boundedness of $q$ is established by the boundedness of $q$  (\ref{eq:q_boundedness}) and the boundedness of $u^{\Delta}$ (\ref{eq:bounded_pw_u}). 
By the Lipschitz continuity of $q$ (\ref{eq:q_Lip_continuity}), 
 \begin{eqnarray*}
&& \left|\int_{0}^{T}\int_{a-\delta}^{b+\delta}q(u^{\Delta}(x),u^{\Delta}(x+h^{\Delta }(h)))-q(u^{*}(x),u^{*}(x+h))
 dx dt\right|\\
 &&\quad \leq \int_{0}^{T}\int_{a-\delta}^{b+\delta}|q(u^{\Delta}(x),u^{\Delta}(x+h^{\Delta }(h)))-q(u^{*}(x),u^{*}(x+h^{\Delta}(h)))|
 dx dt\\
 &&\qquad +\int_{0}^{T}\int_{a-\delta}^{b+\delta}|q(u^{*}(x),u^{*}(x+h^{\Delta}(h)))-q(u^{*}(x),u^{*}(x+h))|
 dx dt\\
  &&\quad \leq C\int_{0}^{T}\int_{a-\delta}^{b+\delta}|u^{\Delta}(x)-u^{*}(x)|+
 |u^{\Delta}(x+h^{\Delta }(h))-u^{*}(x+h^{\Delta }(h))|
 dx dt\\
&&\qquad+ C\int_{0}^{T}\int_{a-\delta}^{b+\delta}|u^{*}(x+h^{\Delta }(h))-u^{*}(x+h)|
 dx dt  := I+II.
  \end{eqnarray*}
 
 For the first integral, assumption  (\ref{eq: converge_subsequence_nonlocal})   yields 
  \begin{eqnarray*}
&& \lim_{\Delta x\rightarrow 0}I 
 =\lim_{\Delta x\rightarrow 0}\int_{0}^{T}\int_{a-\delta}^{b+\delta}|u^{\Delta}(x)-u^{*}(x)|dxdt
+\lim_{\Delta x\rightarrow 0}\int_{0}^{T}\int_{a-\delta+h^{\Delta}(h)}^{b+\delta+h^{\Delta}(h)}
 |u^{\Delta}(y)-u^{*}(y)|
 dy dt\\
  &&\quad \leq\lim_{\Delta x\rightarrow 0}\int_{0}^{T}\int_{a-\delta}^{b+\delta}|u^{\Delta}(x)-u^{*}(x)|dxdt
 +\lim_{\Delta x\rightarrow 0}\int_{0}^{T}\int_{a-\delta}^{b+\delta+h}
 |u^{\Delta}(y)-u^{*}(y)|
 dy dt=0,\quad\quad 
   \end{eqnarray*}
   with the last equality coming from (\ref{eq: converge_subsequence_nonlocal}). 
 For the second integral, 
by the boundedness of $u^{*}$ (\ref{eq:bded_continuity_u*}),   we can apply  dominated convergence theorem:
 \begin{eqnarray*}
\lim_{\Delta x\rightarrow 0} II&=& \int_{0}^{T}\int_{a-\delta}^{b+\delta}\lim_{\Delta x\rightarrow 0}|u^{*}(x+h^{\Delta }(h))-u^{*}(x+h)|
 dx dt= 0, 
\end{eqnarray*}
the last equality is due to the (almost everywhere) continuity of $u^{*}$ (\ref{eq: approx_e2}). 
So we get the condition (ii) which completes the proof of this lemma.
\end{proof}


\subsection{Existence theory}

Based on the proof of Theorem \ref{thm_convergence_local_nonlocal_entropy_solution}, we also
obtain the existence of the entropy solution $u^{\delta}$ for the nonlocal model  (\ref{eq:Cauchy}).  Note that
the uniqueness of the nonlocal entropy solution can be found in Theorem \ref{thm_uniqueness}. These statements
are summarized in the following theorem.

\begin{theorem}\label{corollary_existence_nonlocal_entropy_solution}{{\bf(Well-posedness theory for the nonlocal  model)}}
Consider the nonlocal conservation law (\ref{eq:Cauchy}) and
 assume that $g$ satisfies the monotonicity and regularity conditions 
 (\ref{eq:g_consistent}), (\ref{eq:g_Lip}), and (\ref{eq:g_monotone}) and that the kernel 
$\omega^{\delta}$  satisfies the condition (\ref{eq:w_density}).  
Then, for every initial data $u_0\in L^{1}(\R)\cap L^{\infty}(\R)$, there exists a unique entropy solution $u^{\delta}$ to (\ref{eq:Cauchy}) with $u(0)=u_0$ and, moreover, the following properties hold: 

(a) \quad If $u_0\in BV(\R)$, then the map $t\mapsto u^{\delta}(\cdot,t)$ is Lipschitz continuous with values in $L^{1}(\R)$ and
$$
 ||u^{\delta}(\cdot,t)||_{BV(\R)}\leq ||u_0||_{BV(\R)},  \qquad t \geq 0.
$$

(b) \quad For any two solutions $u^{\delta}, v^{\delta}$ and initial data $u_0, v_0$, the contraction property holds: 
$$
 ||u^{\delta}(\cdot, t)-v^{\delta}(\cdot, t)||_{L^1(\R)}\leq ||u_0-v_0||_{L^{1}(\R)}, \qquad t \geq 0.
$$ 

(c) \quad Monotonicity property: the inequality $u_0\leq v_0$ implies $u^{\delta}\leq v^{\delta}$.

(d) \quad Maximum principle: the inequalities $a\leq u_0\leq b$ imply $a\leq u^{\delta}\leq b$.
\end{theorem}


\section*{Acknowledgment}

The work of the first author (QD) was supported in part by the NSF grant
DMS-1558744, the AFOSR MURI Center for material failure prediction through peridynamics,
and the ARO MURI Grant W911NF-15-1-0562. This work was done when the third author (PLF) enjoyed the hospitality of the Courant Institute of Mathematical Sciences at New York University.


\end{document}